\documentclass[a4paper,12pt]{amsart}
\usepackage{amsmath,amsfonts,amssymb,amsthm}
\usepackage{latexsym,graphicx}
\usepackage{xypic}
\usepackage[matrix,arrow,ps,color,line,curve,frame]{xy}
\usepackage{tikz}
\usepackage{cancel}
\usepackage{comment}
\usepackage{hyperref}
\usepackage{color, colortbl}
	
\definecolor{Gray}{gray}{0.8}
\definecolor{LGray}{gray}{0.91}

\newtheorem{proposition}{\sc Proposition}[section]
\newtheorem{lemma}[proposition]{\sc Lemma}
\newtheorem{corollary}[proposition]{\sc Corollary}
\newtheorem{theorem}[proposition]{\sc Theorem}
\newtheorem{definition}[proposition]{\sc Definition}
\theoremstyle{definition}

\theoremstyle{remark}
\newtheorem{remark}[proposition]{\sc Remark}

\renewcommand{\phi}{\varphi}
\renewcommand{\epsilon}{\varepsilon}
\renewcommand{\subset}{\subseteq}

\renewcommand{\[}{\begin{equation}}
\renewcommand{\]}{\end{equation}}

\newlength{\bibitemsep}
\newlength{\bibparskip}
\let\oldthebibliography\thebibliography
\renewcommand\thebibliography[1]{%
  \oldthebibliography{#1}%
  \setlength{\parskip}{\bibitemsep}%
  \setlength{\itemsep}{\bibparskip}%
}

\begin{document}
\numberwithin{proposition}{section}
\numberwithin{equation}{section}
\baselineskip=13pt
\author{Piotr~M.~Hajac}
\address{Instytut Matematyczny, Polska Akademia Nauk, ul.~\'Sniadeckich 8, Warszawa, 00--656 Poland} 
\email{pmh@impan.pl}
\author{Elizabeth ~A. ~Pacheco}
\address{Centre for Research in Mathematics and Data Science, Western Sydney University, Australia}
\email{e.pacheco@westernsydney.edu.au}
\author{\L ukasz~Kaczmarczyk}
\address{Wydział Matematyki, Informatyki i Mechaniki, Uniwersytet Warszawski, ul. Banacha 2,Warszawa 02--097 Poland}
\email{lw.kaczmarczyk@uw.edu.pl}
\author{Mateusz~Lowiel}
\address{Wydział Matematyki, Informatyki i Mechaniki, Uniwersytet Warszawski, ul. Banacha 2,Warszawa 02--097 Poland}
\email{m.lowiel@uw.edu.pl}
\title[The maximal dimensions of path and graph algebras]{{\large The maximal dimensions of\\ path and graph algebras}}
\maketitle
\vspace*{-10mm}\begin{abstract}
We consider the class of acyclic connected directed graphs with $N\geq 1$. In this paper we find the optimal upper bound for the number of paths amongst acyclic, connected graphs with $N$ edges. We prove that it is in fact optimal by finding an acyclic, connected graph with $N$ edges that realizes this bound. We then adapt these methods to find an optimal bound for Leavitt path algebras over a finite, acyclic, connected graph with $N$ edges.
\end{abstract}

\tableofcontents

\section{Introduction}
\noindent 
Graph theory is considered one of the oldest and most accessible branches of combinatorics and has numerous 
natural connections to other areas of mathematics. 
In particular, directed graphs, or quivers, are fundamental tools in representation 
theory \cite{ass06} as well as in noncommutative geometry \cite{pr06}
and topology~\cite{hrt20,cht21}. 
In this paper, we focus entirely on the combinatorics of directed graphs: optimization and counting problems are common throughout combinatorics, 
and our paper solves one of them
(cf.~\cite{ssz23}).

Inspired by Theorem~\ref{oldthm}, proven in \cite{ht19}, in this paper we investigate the maximal dimension of a finite-dimensional path algebra over a connected quiver with a given number of edges. This is a natural follow-up to the problem solved in \cite{ht19}. We find the optimal upper bound for the number of paths in an acyclic, connected graph and prove that it is in fact optimal by finding a graph with the largest number of paths among acyclic, connected graphs. To achieve this, we employ multiple moves that increase the number of paths, which allow us to consider more and more specialized graphs. Asymptotically we end up with exactly one candidate, up to a mirror image symmetry. The formulas for the optimal upper bound given a number of edges is proven in Theorem~\ref{3sis}. The problem of optimal path algebras is proven in Section 3.

As it turns out, by employing some adaptations one can find the optimal upper bound for the maximal dimension of a Leavitt path algebra over a finite, acyclic, connected graph with a fixed number of edges. A representative for a graph whose Leavitt path algebra is of the maximal possible dimension is given in Theorem~\ref{th4.0}. The problem of optimal Leavitt path algebras is proven in Section 4.

We end the paper off with an interesting follow-up opposite problem of realizing a matrix algebra by a minimal amount of edges, see Corollary~\ref{necessary} and Remark~\ref{Remark}.
\section{Preliminaries}
\begin{definition}
A directed graph is a quadruple $E:=(E^0,E^1,s,t)$ consisting of
\begin{enumerate}
\item
 the set of vertices $E^0$ and the set of edges $E^1$,
\item 
the source and target maps $s_E,t_E\colon E^1\rightarrow E^0$
assigning to each edge its source and target vertex, respectively,
\end{enumerate}
\end{definition}

A finite path in a directed graph $E$ is a finite tuple $p_n:=(e_1,\ldots,e_n)$ of edges
satisfying 
\[
t_E(e_1)=s_E(e_2),\quad t_E(e_2)=s_E(e_3),\quad \ldots,\quad t_E(e_{n-1})=s_E(e_n).
\]
The beginning $s(p_n)$ of $p_n$
is $s(e_1)$ and the end $t(p_n)$ of $p_n$ is $t(e_n)$. If $s(p_n)=t(p_n)$, we call $p_n$ 
a loop, or a cycle. A directed graph without loops is called acyclic.
The length of a path is the size of the tuple. Every edge is a path of length~$1$, and
vertices are considered as finite paths of length~$0$. The set of all paths in $E$ of length $k$ is denoted by~$FP_{k}(E)$, and the set of all finite 
paths in $E$ by~ $FP(E)$. Undirected paths are paths in which the condition 
 $t_E(e)=s_E(f)$ whenever the edge $f$ follows the edge $e$ is replaced by the requirement that 
$\{s_E(e),t_E(e)\}\cap\{s_E(f),t_E(f)\}\neq\emptyset$ whenever the edge $f$ follows the edge~$e$. We say that a directed
 graph is connected iff
any two vertices are connected by a finite undirected path.

This paper is inspired by the results of \cite{ht19}, where an optimal upper bound for the number of paths of fixed length was found. We give the precise formulation below.
\begin{theorem}[\cite{ht19,c-a20}]\label{oldthm}
Let $E$ be an acyclic directed graph with $N\geq 1$ edges, and let 
\[\nonumber
1\leq k\leq N=:nk+r,\quad
0\leq r\leq k-1.
\]
Then there are at most
\[\nonumber
\boxed{
P^N_k:=(n+1)^rn^{k-r}
}
\]
different paths of length $k$, and the bound is optimal.
\end{theorem}
\section{Path algebras}
\noindent
\subsection{From forests to trunks}
Let $E$ be a connected acyclic graphs with $N>0$ edges. We denote the set of all connected acyclic graph with $N>0$ edges by~$\mathcal{E}^N$, and we write 
\begin{equation}
P^N\!(E):=|\mathrm{FP}(E)|
\end{equation}
for the number of all paths in~$E$. 
\begin{definition}
Let $(E^0,\,E^1,\,s_E,\,t_E)$ be an acyclic graph with $N>0$ edges satisfying $E^0=s_E(E^1)\cup t_E(E^1)$. 
We call such graphs \emph{$N$-forest graphs}. 
A connected $N$-forest graph is called an \emph{$N$-tree graph}, and an $N$-tree graph
is referred to as  a \emph{trunk graph} when its vertices can be ordered $v_1,\ldots, v_n$ in such a way that
any edge $e\in E^1$ enjoys the property $s_E(e)=v_i$ and $t_E(e)=v_{i+1}$ for some~$1\leq i\leq n-1$:
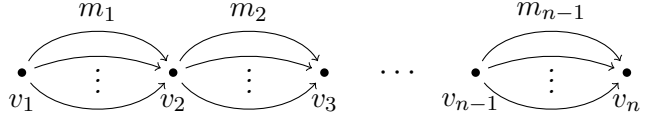
\begin{center}
\begin{tikzpicture}[scale=4]

\fill (0,0)  circle[radius=.4pt];
\draw (0,-.1) node{ {\small $v_1$ }};
\draw[->, shorten >=5pt, shorten <=5pt,out=60, in=120] (0,0) to (.5,0);
\draw[->, shorten >=5pt,shorten <=5pt,out=20, in=160] (0,0) to (.5,0);
\draw[->, shorten >=5pt, shorten <=5pt,out=-50, in=-130] (0,0) to (.5,0);
\draw (.25,0) node{$\vdots$};
\draw (.25,.2) node{{\small $m_1$ }};
\fill (.5,0)  circle[radius=.4pt];
\draw (.5,-.1) node{ {\small $v_2$ }};
\draw[->, shorten >=5pt, shorten <=5pt,out=60, in=120] (.5, 0) to (1,0);
\draw[->, shorten >=5pt,shorten <=5pt,out=20, in=160] (.5, 0) to (1,0);
\draw[->, shorten >=5pt, shorten <=5pt,out=-50, in=-130] (.5, 0) to (1,0);
\draw (.75,0) node{$\vdots$};
\draw (.75,.2) node{{\small $m_2$ }};
\fill (1,0)  circle[radius=.4pt];
\draw (1,-.1) node{ {\small $v_3$ }};
\draw (1.25,0) node{$\cdots$};
\fill (1.5,0)  circle[radius=.4pt];
\draw (1.5,-.1) node{ {\small $v_{n-1}$ } };
\draw[->, shorten >=5pt,shorten <=5pt, out=60, in=120] (1.5, 0) to (2,0);
\draw[->, shorten >=5pt,shorten <=5pt,out=20, in=160] (1.5, 0) to (2,0);
\draw[->, shorten >=5pt, shorten <=5pt, out=-50, in=-130] (1.5,0) to (2,0);
\draw (1.75,0) node{$\vdots$};
\draw (1.75,.2) node{ {\small$m_{n-1}$ }};
\fill (2,0)  circle[radius=.4pt];
\draw (2,-.1) node{ {\small $v_n$ }};
\end{tikzpicture}.
\end{center}
We denote the set of all forest, tree and trunk graphs with  $N>0$ edges by $\mathcal{E}^N$, $\mathcal{F}^N$, and $\mathcal{T}^N$, respectively.
The trunk graph as in the above picture will be written as $(m_1,\ldots,m_{n-1})$.
\end{definition}

\begin{definition}
Let $E$ be a trunk graph with $N>0$ edges and $n$ vertices. The \emph{width} of $E$ is
$$
w(E):=\max_{v\in E^0}|s_E^{-1}(v)|,
$$ 
 the \emph{weight} of $E$ is~$N$, and the \emph{length} of $E$ is~$n-1$. 
 We denote the set of all trunk graphs with  $N>0$ edges and the width $w$ by~$\mathcal{T}^N_w$.
\end{definition}

\begin{lemma}\label{lem0}
The following equality holds:
$$
\sup_{E\in\mathcal{E}^N}P^N\!(E)=\sup_{E\in\mathcal{T}^N}P^N\!(E).
$$
\end{lemma}
\begin{proof}
Let $E\in\mathcal{E}^N$. Our proof consists of the following three constructions, all of which can only enlarge the value $P^N(E)$:
\begin{enumerate}
	\item \emph{Splitting sinks and sources}. For any sink $v\in E^0$ such that $|t_E^{-1}(v)|\ge2$, consider the graph $F=(F^0,F^1,s_F,t_F)$, whose set of vertices is
	\begin{equation*}
		F^0 := \left(E^0\setminus\{v\}\right)\sqcup \{v_e:e\in t_E^{-1}(v)\}.
	\end{equation*}
	Both the set of edges $F^1:=E^1$ and the source map $s_F:=s_E$ we set to be the same. The target map $t_F$ is defined as follows
	\begin{equation*}
		t_F(e) = \begin{cases}
			v_e & \text{ if }e\in t_E^{-1}(v),\\
			t_E(e) & \text{ otherwise.}
		\end{cases}
	\end{equation*}
	Clearly $F\in\mathcal{E}^N$. Note that $P^N(F)> P^N(E)$ as $F$ contains more vertices than $E$ while preserving the amount of paths of positive length. We say that the graph $F$ came from the graph $E$ by splitting of the sink $v$. Dually, one can split a source and that construction also can only increase the total amount of paths.
	\item \emph{Combing stray paths}. Suppose that inside $E$ we have two paths $e=(e_1,\dots,e_n)$ and $v=(v_1,\dots,v_k)$ such that there exists $i\in\{1,\dots,n\}$ for which we have $t(v_k)=t(e_i)$ and $s_E(v_j)$ is never a vertex that $e$ runs through. We can define a new graph $G=(G^0,G^1,s_G,t_G)$, where the set of vertices, the set of edges and the source map stays the same $G^0:=E^0, G^1:=E^1, s_G:=s_E$. The target map $t_G$ differs from $t_E$ only on $v_k$: 
	\begin{equation*}
		t_G(e) = \begin{cases}
			s_E(e_1) & \text{ if }e=v_k,\\
			t_E(e) & \text{ otherwise}.
		\end{cases}
	\end{equation*} 
	Clearly $G$ is connected has the same amount of edges as $E$ and our assumption of all $s_E(v_j)$ never being a vertex that $e$ runs through ensures that $G$ remains acyclic. Thus $G\in\mathcal{E}^N$ and note that $P^N(G)> P^N(E)$ as we can define an injection $FP(E)\to FP(G)$ that sends any path $w$ that does not contain $v_k$ to itself -- note that our graphs are completely the same if we cut out this particular edge, so this is a well-defined mapping. If $w$ contains $v_k$, say $w=(w_1,\dots,w_m)$ with $v_k=w_j$ for some $j\in\{1,\dots,m\}$, then we shall send $w$ to the following path
	\begin{equation*}
		w'=(w_1,\dots,w_j,e_1,\dots,e_i,w_{j+1},\dots,w_m).
	\end{equation*}
	In other words, we put $v_k,e_1,\dots,e_i$ instead of $v_k$ in $w$. Note that this mapping is clearly not surjective -- for example $(v_k,e_1)$ will always be a path that will not be attained by this map. Similarly as in the first procedure we can dualize this construction to paths that start on the fixed path $(e_1,\dots,e_n)$. Of course this dual construction also increases the total amount of paths.
	\item \emph{Resolving alternative paths}. Suppose that we two paths $e=(e_1,\dots,e_n),v=(v_1,\dots,v_m)$, with $m\ge n>1$, such that $s_E(e_1)=s_E(v_1)$ and $t_E(e_n)=t_E(v_m)$, and there is no path connecting vertices that $e$ and $v$ run through. We want to merge vertices that $v$ runs through with vertices that $e$ runs through -- this operation reduces the total amount of vertices, but the total amount of paths increases. Explicitly, we construct a graph $H=(H^0,H^1,s_H,t_H)$ as follows: Set the vertices to be
	\begin{equation*}
		H^0 := E^0 \setminus \{ t(v_k):k=1,\dots,m-1 \}
	\end{equation*}
	and the edges remain the same $H^1:=E^1$. We change the source and target maps to reflect the process of merging together our desired vertices
	\begin{align*}
		& s_H(e) := \begin{cases}
			s_E(e_i) & \text{ if } e=v_i \text{ for some }i\in\{ 1,\dots,m \},\\
			s_E(e) & \text{otherwise},
		\end{cases}\\
		& t_H(e) := \begin{cases}
			t_E(e_i) & \text{ if } e=v_i \text{ for some }i\in\{ 1,\dots,m \},\\
			t_E(e) & \text{otherwise}.
		\end{cases}
	\end{align*}
	Our final graph $H$ clearly is connected and has the same amount of edges that $E$ does and our assumption of having no paths connecting that $v$ and $e$ run through we get that $H$ must still be acyclic. Thus $H\in\mathcal{E}^N$. In order to see that $P^N(H)>P^N(E)$ observe that any path of positive length $w$ in $E$ is still a valid path in $H$, but we get more -- in our new graph $H$ we can mix the $v_i$'s with $e_j$'s, hence we get at least $2^m-2$ new edges, while losing $m-1$ vertices, which is a net positive under our assumption $m>1$.
\end{enumerate}
Having all of these moves defined we modify our graph $E$ in the following way. First, we split all sinks and sources and let $F$ be the final product of all of these 
procedures. The graph $F$ may be disconnected and we shall proceed by considering every connected component of $F$ separately -- hence let us assume that 
$F$ is connected. Next, we can identify all instances of a path splitting into multiple paths and merges together later on. Suppose that we have two paths $e=(e_1,\dots,e_n)$ and $v=(v_1,\dots,v_m)$. It might be the case that there exists a path connecting two vertices that $v$ and $e$ runs through, say $w$ is a path such that $s(w)=s_E(v_i)$ and $t(w)=t_E(e_j)$ for some $i=1,\dots,m-1$ and $j=2,m$. Therefore we can first consider the splitting of paths that certainly occurs at $s_E(v_i)$, since $(w,e_{j+1},\dots,e_n)$ (or in the case $j=m$, just $w$ by itself) and $(v_i,\dots,v_m)$ are two paths with the same source and target. Note that if we resolve the alternative routes which are $(w,e_{j+1},\dots,e_n)$ and $(v_i,\dots,v_m)$, then the paths $(v_1,\dots,v_{i-1})$ and $e_1,\dots,e_j$ form a new splitting of paths, provided that $i>1$, which we can take care of afterwards. We can therefore iteratively resolve all alternative path obtaining a new graph that we shall denote by $G$. The graph $G$ is now just a collection of paths that cross each other in a way that is acyclic. We can therefore pick any path in $G$ and comb all stray paths that are either going into and outgoing from our chosen graph. Thus the end product is a graph that lies in $\mathcal{T}^N$ and since all of our procedures only increase the total amount of paths, then our final product also has a larger total amount of paths. Thus for any graph $E\in\mathcal{E}^N$ we can find a better graph that lies in $\mathcal{T}^N$, which finishes the proof.
\end{proof}

\begin{proposition} \label{22n}
$|\mathcal{T}^N|=2^{N-1}$.
\end{proposition}
\begin{proof} 
Note that $|\mathcal{T}^N|=|\mathcal{T}^{N-1}|+|\mathcal{T}^{N-1}|$ as we can obtain an element of $\mathcal{T}^{N}$ from an element of 
$\mathcal{T}^{N-1}$ in the following two ways: by appending an arrow to the beginning or adding an additional arrow between the first and second vertex. 
These constructions yield disjoint distributions of $m_1$ though $m_k$ (those where $m_1=1$ and those $m_1\geq 2$). These account for all elements of 
$\mathcal{T}^N$. Thus, as  $|\mathcal{T}^1|=1$, we get that $|\mathcal{T}^N|=2^{N-1}$.
\end{proof}

Note that one can also prove the above proposition by observing that the number of all partitions of the set with $m$ elements into $n$ subsets 
(empty subsets allowed) equals $\binom{m+n-1}{m}$ and the number of trunks of  weight $N$ and length $N-k$ is exactly the number of all partitions of the 
set of $k$ edges into $N-k$ many subsets. Consequently,
\begin{equation}
|\mathcal{T}^N|=\sum_{k=0}^{N-1}\binom{N-1}{k}=2^{N-1}.
\end{equation}

\newpage 
\subsection{Experimental data}
We used Python  to find out which trunks of weight $N$ maximize the number of all paths.
However, as indicated by Proposition~\ref{22n}, this required checking $2^{N-1}$ cases, so the calculations for $N>25$ were taking too long.

\phantom{.}\\
\begin{tabular}{lllll}
\rowcolor{Gray}
$N$ edges         & Maximal number of paths   & Optimal trunks  &  Length of longest path \\
\rowcolor{LGray}
3&3&(1,1,1)&3\\

4&15&(1,1,1,1)&4\\
\rowcolor{Gray}
5&21&(1,1,1,1,1), (1,1,2,1),(1,2,1,1)&4 or 5\\
\rowcolor{LGray}
6&31&(1,2,2,1)&4\\

7&44&(1,1,2,2,1),(1,2,2,1,1)&5\\
\rowcolor{Gray}
8&66&(1,2,2,2,1)&5\\
\rowcolor{LGray}
9&91&(1,1,2,2,2,1),(1,2,2,2,1,1),(1,2,3,2,1)& 5 or 6\\

10&137&(1,2,2,2,2,1)&6\\
\rowcolor{Gray}
11&192&(1,2,2,3,2,1),(1,2,3,2,2,1)&6\\
\rowcolor{LGray}
12&280&(1,2,2,2,2,2,1)&7\\

13&401&(1,2,2,3,2,2,1)&7\\
\rowcolor{Gray}
14&571&(1,2,2,3,3,2,1),(1,2,3,3,2,2,1)&7\\
\rowcolor{LGray} 
15&820&(1,2,2,2,3,2,2,1),(1,2,2,3,2,2,2,1)&8 \\ \hline \hline
16&1194&(1,2,2,3,3,2,2,1)&8\\
\rowcolor{Gray}
17&1709&(1,2,2,3,3,3,2,1),(1,2,3,3,3,2,2,1)&8\\
\rowcolor{LGray}
18&2449&(1,2,3,3,3,3,2,1)&8\\

19&3574&(1,2,2,3,3,3,2,2,1)&9\\
\rowcolor{Gray}
20&5124&(1,2,2,3,3,3,3,2,1),(1,2,3,3,3,3,2,2,1)&9\\
\rowcolor{LGray}
21&7349&(1,2,3,3,3,3,3,2,1)&9\\

22&10715&(1,2,2,3,3,3,3,2,2,1)&10\\
\rowcolor{Gray}
23&15370&(1,2,2,3,3,3,3,3,2,1),(1,2,3,3,3,3,3,2,2,1)&10\\
\rowcolor{LGray}
24&22050&(1,2,3,3,3,3,3,3,2,1)&10\\
25&32139&(1,2,2,3,3,3,3,3,2,2,1)&11\\

\end{tabular}

\subsection{Main result}
\begin{lemma}\label{lem1}
The following equality holds:
$$
\sup_{E\in\mathcal{T}^N}P^N\!(E)=\sup_{E\in\mathcal{T}^N_3}P^N\!(E).
$$
\end{lemma}
\begin{proof}
Let $E\in\mathcal{T}^N$ and suppose that we have an edge $e\in E^1$ which has a label $n>3$ -- here we consider $E$ as a simple and labeled graph with labels coming from $\mathbb{N}$. We can blow such an edge up into two edges $e_1,e_2$ such that the label of $e^1$ is $\lfloor n/2\rfloor$ and the label of $e^2$ is $\lceil n/2 \rceil$. Explicitly, we define a graph $F=(F^0,F^1,s_F,t_F)$ as follows: Let the set of vertices and edges to be
\begin{align*}
	F^0 := E^0\sqcup\{\epsilon\}, && F^1 = \left(E^1\setminus \{e\}\right)\sqcup \{e^1,e^2\}.
\end{align*}
Next, we need to define the source and target maps:
\begin{align*}
	& s_F(w) := \begin{cases}
		s_E(e) & \text{ if }w=e^1,\\
		\epsilon & \text{ if }w=e^2,\\
		s_E(w) & \text{ otherwise}.
	\end{cases}\\
	& t_F(w) := \begin{cases}
		\epsilon & \text{ if }w=e^1,\\
		t_E(e) & \text{ if }w=e^2,\\
		t_E(w) & \text{ otherwise}.
	\end{cases}\
\end{align*}
Lastly, we need the labeling $l_F:F^1\to\mathbb{N}$ and it is defined as follows
\begin{equation*}
	l_F(w) := \begin{cases}
		\left\lfloor \frac{n}{2}\right\rfloor &\text{ if }w=e^1,\\
		\left\lceil \frac{n}{2} \right\rceil & \text{ if }w=e^2,\\
		l_E(w) &\text{ otherwise}.
	\end{cases}
\end{equation*}
Clearly any path in $E$ that does not contain $e$ corresponds to a unique path in $F$ that does not contain either $e^1$ or $e^2$. Let $L$ be the amount of paths 
in $E$ that end at $s_E(e)$ and $R$ be the amount of paths in $E$ that start at $t_E(e)$. Then the total amount of paths that contain $e$ is 
\begin{equation*}
	T_E = LRn.
\end{equation*}
On the other hand, the total amount of paths in $F$ that contain either $e^1$ or $e^2$ is
\begin{equation*}
	T_F +1= LR\left\lfloor \frac{n}{2}\right\rfloor\left\lceil \frac{n}{2} \right\rceil + L\left\lfloor \frac{n}{2}\right\rfloor + R\left\lceil \frac{n}{2} \right\rceil.
\end{equation*}
As $n\ge4$, it is clear that the first term of $T_F$ already overpowers $T_E$, thus our construction can only increase the total number of paths. Iterating this for all 
edges with label at least $4$ we get that for any graph $E\in\mathcal{T}^N$ we can find a better candidate that lies in $\mathcal{T}_3^N$, which finishes the proof.
\end{proof}
\begin{lemma}\label{lem2}
Let $F\in \mathcal{T}^N_3$. Then, if $\sup_{E\in\mathcal{T}^N_3}P^N\!(E)=P^N\!(F)$, we have
\begin{equation}\label{3,1}
\forall\; i\in\{1,\ldots,n-1\}\colon |s_F^{-1}(v_i)-s_F^{-1}(v_{i+1})|\leq 1,
\end{equation}
where $F^0=\{v_1,\dots,v_n\}$.
\end{lemma}
\begin{proof}
Given the fact that we are working with trunk graphs with at most $3$ edges between two consecutive vertices the condition $\eqref{3,1}$ boils down to proving that we must not have a single edge followed by $3$ edges (or vice versa). Consider the following construction:

Let $E\in\mathcal{T}_3^N$. Suppose that in our graph $E$ we have three vertices $v_1,v_2,v_3$ such that there is a single edge $e_1$ connecting $v_1$ and $v_2$, and three edges connecting $v_2$ and $v_3$, say $e_2,e_3,e_4$. Consider the graph $H=(H^0,H^1,s_H,t_H)$ that is the transfer of $e_2$ to vertices $v_1$ and $v_2$, i.e. we set $H^0:=E^0, H^1:=E^1$ and we change the source and target functions accordingly:
\begin{align*}
	s_H(e) := \begin{cases}
		v_1 & \text{ if }e = e_2,\\
		s_E(e) & \text{ otherwise},
	\end{cases}\\
	t_H(e) := \begin{cases}
		v_2 & \text{ if }e = e_2,\\
		t_E(e) & \text{ otherwise}.
	\end{cases}
\end{align*}
Clearly $H\in\mathcal{T}_3^N$ and observe that we have a bijection of paths not going through any of vertices $v_1,v_2,v_3$ between $E$ and $H$. Let us denote by $L\ge1$ the amount of paths ending at $v_1$ and by $R\ge1$ the amount of paths ending at $v_3$. The total amount of paths in $E$ that go through either $v_1,v_2$ or $v_3$ is given by
\begin{equation*}
	T_E = L + 3LR,
\end{equation*}
whereas the total amount of paths in $H$ that go through either of these vertices is given by
\begin{equation*}
	T_H = 2L + 4LR.
\end{equation*}
As $L\ge1$, then we see that $T_H\ge T_E$, thus this construction can only increase the total amount of paths in our graph. Of course this works the same when considering a situation where we have first $3$ edges and then a single edge. In particular this means that our optimal graph $F$ must not have such pairs of consecutive amounts of edges.
\end{proof}
Let us denote the set of all graphs in $\mathcal{T}^N_3$ satisfying \eqref{3,1} by~$\mathcal{T}^N_{3,1}$.
\begin{lemma}\label{lem3}
Let $F\in \mathcal{T}^N_{3,1}$ be such that $\sup_{E\in\mathcal{T}^N_{3,1}}P^N\!(E)=P^N\!(F)$. Then $|s_F^{-1}(v_i)|=1$ for  $i=1,n-1$, where $F^0=\{v_1,\dots,v_n\}$.
\end{lemma}
\begin{proof}
Suppose that for the first vertex $v_1\in F^1$ we have $s_F^{-1}(v_1)=\{e_1,\dots,e_k\}$ with $k>1$. Consider the graph $F'=((F')^0,(F')^1,s_{F'},t_{F'})$ defined as follows: 
\begin{align*}
	& (F')^0 := F\sqcup\{ v^i:i=2,\dots,k \},\\
	& (F')^1 := F^1,\\
	& s_{F'}(e) := \begin{cases}
		v^{i} &\text{ if }e = e_i\text{ for some }i=2,\dots,k,\\
		s_{F}(e) &\text{ otherwise}.
	\end{cases}\\
	& t_{F'}(e) := \begin{cases}
		v^{i-1} &\text{ if }e=e_i \text{ for some }i=3,\dots,k,\\
		v_1 &\text{ if }e=e_2,\\ 
		t_{F}(e) & \text{ otherwise}.
	\end{cases}
\end{align*}
Graphically, $F'$ comes from taking the leftmost edges of $F$ and 'spreading them' to the left, extending the graph to the left. There is a clear injection $PF\hookrightarrow PF'$ that sends a path $w$, which does not contain any of the $e_i$'s, in $F$ to the same path $w$, considered as a path in $F'$. Any path $w=(e_i,\dots)$ that starts with an edge $e_i$ with $i\ge1$ we shall send to the appropriate extension:  
\begin{equation*}
	w' = (e_i,e_{i-1},\dots,e_1,\dots).
\end{equation*}
As $k>1$, then this injection is not surjective as $PF'$ contains strictly more vertices and any new vertex in $F'$ is not attained by our mapping. A dual argument implies that $F$ necessarily has to have only one edge pointing to the unique sink of $F$.
\end{proof}

Given a graph $E\in  \mathcal{T}^N_3$ with $E^0=\{v_1,\dots ,v_n\}$ let $E_{L,m}$ be the full subgraph of $E$ generated by vertexes $\{v_1,\dots ,v_{m-1} \}$ and let $L_m$ denote the number of paths in $E_{L,m}$ ending in vertex $v_{m-1}$.
Note that $L_1=0$, $L_{m+1}=1+\vert s^{-1}(v_{m-1})\vert L_m$, therefore for any $k<m$ the inequality $L_k < L_m$ holds.
Analogously let $E_{R,m}$ be the full subgraph of $E$ generated by vertexes $\{v_{m+1},\dots ,v_n\}$ and let $R_m$ denote the number of paths in $E_{R,m}$ beggining in vertex $v_{m+1}$.
Note that $R_n=0$, $R_{m-1}=1+\vert s^{-1}(v_m) \vert R_m$, therefore for any $k<m$ the inequality $R_m < R_k$ holds.

\begin{lemma}\label{lem4}
Let $F\in \mathcal{T}^N_3$. If $\sup_{E\in\mathcal{T}^N_3}P^N\!(E)=P^N\!(F)$, then for $k=\min\{ m\colon L_{m+1} \ge R_{m+1} \}$ the following inequalities hold:
\begin{displaymath}
\forall_{1\le i < k} \vert s^{-1}(v_i) \vert \le \vert s^{-1}(v_{i+1}) \vert \ \textrm{ and } \ \forall_{k\le i \le n-2} \vert s^{-1}(v_{i+1}) \vert \le \vert s^{-1}(v_i) \vert
\end{displaymath}
where $F^0=\{v_1,\dots,v_n\}$.
\end{lemma}
\begin{proof}
First note that $L_n\ge n$ and $R_{n}=0$ therefore the set $\{ m\colon L_{m+1} \ge R_{m+1} \}$ is nonempty.
Moreover for $i>k$ the inequalities $L_{i+1}\ge L_{k+1}\ge R_{k+1}\ge R_{i+1}$ show that $\{ m\colon L_{m+1} \ge R_{m+1} \}=\{m\colon k\le m\}$.
Let us fix $i$, denote  $x=\vert s^{-1}(v_i) \vert$, $\vert s^{-1}(v_{i+1}) \vert=y$ and consider the following decomposition of $F$:
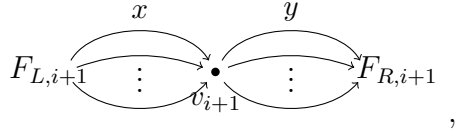
\begin{center}
\begin{tikzpicture}[scale=4]
\draw (-.05,0) node{$F_{L,i+1}$};
\draw[->, shorten >=5pt, shorten <=5pt,out=60, in=120] (0,0) to (.5,0);
\draw[->, shorten >=5pt,shorten <=5pt,out=20, in=160] (0,0) to (.5,0);
\draw[->, shorten >=5pt, shorten <=5pt,out=-50, in=-130] (0,0) to (.5,0);
\draw (.25,0) node{$\vdots$};
\draw (.25,.2) node{{\small $x$ }};
\fill (.5,0)  circle[radius=.4pt];
\draw (.5,-.1) node{ {\small $v_{i+1}$ }};
\draw[->, shorten >=5pt, shorten <=5pt,out=60, in=120] (.5, 0) to (1,0);
\draw[->, shorten >=5pt,shorten <=5pt,out=20, in=160] (.5, 0) to (1,0);
\draw[->, shorten >=5pt, shorten <=5pt,out=-50, in=-130] (.5, 0) to (1,0);
\draw (.75,0) node{$\vdots$};
\draw (.75,.2) node{{\small $y$ }};
\draw (1.1,0) node{$F_{R,i+1}$};
\end{tikzpicture},
\end{center}
and compare it with a graph $G$ given by decomposition:
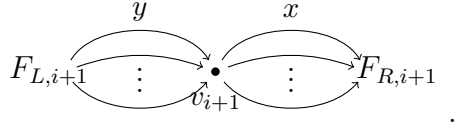
\begin{center}
\begin{tikzpicture}[scale=4]
\draw (-.05,0) node{$F_{L,i+1}$};
\draw[->, shorten >=5pt, shorten <=5pt,out=60, in=120] (0,0) to (.5,0);
\draw[->, shorten >=5pt,shorten <=5pt,out=20, in=160] (0,0) to (.5,0);
\draw[->, shorten >=5pt, shorten <=5pt,out=-50, in=-130] (0,0) to (.5,0);
\draw (.25,0) node{$\vdots$};
\draw (.25,.2) node{{\small $y$ }};
\fill (.5,0)  circle[radius=.4pt];
\draw (.5,-.1) node{ {\small $v_{i+1}$ }};
\draw[->, shorten >=5pt, shorten <=5pt,out=60, in=120] (.5, 0) to (1,0);
\draw[->, shorten >=5pt,shorten <=5pt,out=20, in=160] (.5, 0) to (1,0);
\draw[->, shorten >=5pt, shorten <=5pt,out=-50, in=-130] (.5, 0) to (1,0);
\draw (.75,0) node{$\vdots$};
\draw (.75,.2) node{{\small $x$ }};
\draw (1.1,0) node{$F_{R,i+1}$};
\end{tikzpicture}.
\end{center}
Since 
\begin{align*}
	& \vert FP(F) \vert=\vert FP(F_{L,i+1}) \vert+xL_{i+1}+1+xL_{i+1}yR_{i+1}+yR_{i+1}+\vert FP(F_{R,i+1}) \vert,\\
	& \vert FP(G) \vert=\vert FP(F_{L,i+1}) \vert+yL_{i+1}+1+yL_{i+1}xR_{i+1}+xR_{i+1}+\vert FP(F_{R,i+1}) \vert,
\end{align*}
we see that
\begin{equation*}
\vert FP(F) \vert-\vert FP(G) \vert=xL_{i+1}+yR_{i+1}-yL_{i+1}-xR_{i+1}=(x-y)(L_{i+1}-R_{i+1})\ge 0.
\end{equation*}
The last inequality follows from the condition $P^N\!(F)=\sup_{E\in\mathcal{T}^N_3}P^N\!(E)$.
Therefore if $i<k$ then $L_{i+1}<R_{i+1}$ and for above inequality to hold there must be $\vert s^{-1}(v_i) \vert=x\le y=\vert s^{-1}(v_{i+1}) \vert$ and if $i\ge k$ then $L_{i+1}\ge R_{i+1}$ and for the above inequality to hold there must be $\vert s^{-1}(v_{i+1}) \vert=y\le x=\vert s^{-1}(v_i) \vert$.
\end{proof}
\begin{corollary}\label{edgeseq}
	All lemmas that we have proven thus far yield a very particular form of the sequence of edges $l:=(|s_F(v_i)^{-1}|)_{i=1}^{n-1}$ in an optimal graph. Namely, we know that all entries are either $1$, $2$ or $3$ by Lemma~\ref{lem1}. The difference between two consecutive entries is at most $1$ by Lemma~\ref{lem2}. Further, Lemma~\ref{lem4} implies that the sequence starts by being non-decreasing, reaches a peak and then becomes non-increasing. Finally, we know that it starts and ends with a $1$ by Lemma~\ref{lem3}. So, in fact, $l$ is a sequence of the form
	\begin{equation*}
		l = (1,\dots,1,2,\dots,2,3,\dots,3,2,\dots,2,1,\dots,1).
	\end{equation*}
	Say that there are $n_1$ copies of $1$'s at the beginning, followed by $n_2$ $2$'s, $n_3$ $3$'s, $n_4$ $2$'s and finally $n_5$ $1$'s at the end. We will use the following shorthand for the sequence $l$:
	\begin{equation*}
		l =: (n_1,n_2,n_3,n_4,n_5).
	\end{equation*}
\end{corollary}

\begin{lemma}\label{lem3.1}
Let $F\in \mathcal{T}^N_3$. If $\sup_{E\in\mathcal{T}^N_3}P^N\!(E)=P^N\!(F)$ and $(n_1,n_2,n_3,n_4,n_5)$ be its associated sequence with $n_2>0$ then $n_1,n_5\le 3$.
\end{lemma}

\begin{proof}
Let us assume that $n_1\ge 4$ and decompose the graph:
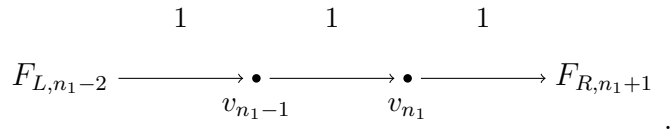
\begin{center}
\begin{tikzpicture}[scale=4]
\draw (-.15,0) node{$F_{L,n_1-2}$};
\draw[->, shorten >=5pt,shorten <=5pt,out=0, in=180] (0,0) to (.5,0);
\draw (.25,.2) node{{\small $1$ }};
\fill (.5,0)  circle[radius=.4pt];
\draw (.5,-.1) node{ {\small $v_{n_1-1}$ }};
\draw[->, shorten >=5pt, shorten <=5pt,out=0, in=180] (.5, 0) to (1,0);
\draw (.75,.2) node{{\small $1$ }};
\fill (1,0)  circle[radius=.4pt];
\draw (1,-.1) node{ {\small $v_{n_1}$ }};
\draw[->, shorten >=5pt, shorten <=5pt,out=0, in=180] (1, 0) to (1.5,0);
\draw (1.25,.2) node{{\small $1$ }};
\draw (1.65,0) node{$F_{R,n_1+1}$};
\end{tikzpicture}.
\end{center}
And let's compare it with graph $G$ whose associated sequence is $(n_1-2,n_2+1,n_3,n_4,n_5)$ and decompose it as
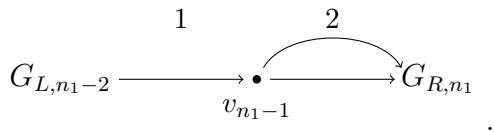
\begin{center}
\begin{tikzpicture}[scale=4]
\draw (-.15,0) node{$G_{L,n_1-2}$};
\draw[->, shorten >=5pt,shorten <=5pt,out=0, in=180] (0,0) to (.5,0);
\draw (.25,.2) node{{\small $1$ }};
\fill (.5,0)  circle[radius=.4pt];
\draw (.5,-.1) node{ {\small $v_{n_1-1}$ }};
\draw[->, shorten >=5pt, shorten <=5pt,out=60, in=120] (0.5, 0) to (1,0);
\draw[->, shorten >=5pt, shorten <=5pt,out=0, in=180] (.5, 0) to (1,0);
\draw (.75,.2) node{{\small $2$ }};
\draw (1.1,0) node{$G_{R,n_1}$};
\end{tikzpicture}.
\end{center}
Since both $F_{L,n_1-2}$ and $G_{L,n_1-2}$ have associated sequence $(n_1-3,0,0,0,0)$ Let us denote the number of paths ending at sinks of these graps by $L$ and similarly since $F_{R,n_1+1}$ and $G_{R,n_1}$ both have associated sequence $(0,n_2,n_3,n_4,n_5)$ let us denote by $R$ the number of all paths starting at the source of these graphs.
Simple computation gives us
\begin{equation*}
\begin{split}
\vert FP(F)\vert&=\vert FP(F_{L,n_1-2})\vert +2L+2R+LR+3+\vert FP(F_{R,n_1+1})\vert \\
\vert FP(G)\vert&= \vert FP(G_{L,n_1-2}) \vert +L+1+2R+2LR+\vert FP(G_{R,n_1})\vert \\
\vert FP(G)\vert-\vert FP(F)\vert&=LR-L-2 \ge 0.
\end{split}
\end{equation*}
The last inequality holds, because $L=n_1-2\ge 2$ and $R\ge 2$ since $n_2>0$.
This shows the optimal graph can ber found among these with $n_1\le 3$.
Dual proof shows the same for $n_5$.
\end{proof}

\begin{lemma}\label{lem5}
Let $F\in \mathcal{T}^N_3$. If $\sup_{E\in\mathcal{T}^N_3}P^N\!(E)=P^N\!(F)$ and $l = (n_1,n_2,n_3,n_4,n_5)$ be its associated sequence and, $n_3>0$ then $\vert n_2-n_4\vert\le 1$.
\end{lemma}
\begin{proof}
Let us assume that $n_2-n_4\ge 2$ and $n_1\ge n_5$ and decompose $F$ as 
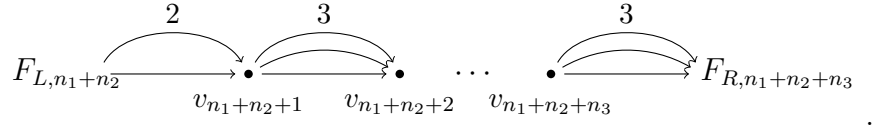
\begin{center}
\begin{tikzpicture}[scale=4]
\draw (-.1,0) node{$F_{L,n_1+n_2}$};
\draw[->, shorten >=5pt, shorten <=5pt,out=60, in=120] (0,0) to (.5,0);
\draw[->, shorten >=5pt,shorten <=5pt,out=0, in=180] (0,0) to (.5,0);
\draw (.25,.2) node{{\small $2$ }};
\fill (.5,0)  circle[radius=.4pt];
\draw (.5,-.1) node{ {\small $v_{n_1+n_2+1}$ }};
\draw[->, shorten >=5pt, shorten <=5pt,out=60, in=120] (.5, 0) to (1,0);
\draw[->, shorten >=5pt,shorten <=5pt,out=30, in=150] (.5, 0) to (1,0);
\draw[->, shorten >=5pt, shorten <=5pt,out=0, in=180] (.5, 0) to (1,0);
\draw (.75,.2) node{{\small $3$ }};
\fill (1,0)  circle[radius=.4pt];
\draw (1,-.1) node{ {\small $v_{n_1+n_2+2}$ }};
\draw (1.25,0) node{$\cdots$};
\fill (1.5,0)  circle[radius=.4pt];
\draw (1.5,-.1) node{ {\small $v_{n_1+n_2+n_3}$ }};
\draw[->, shorten >=5pt, shorten <=5pt,out=60, in=120] (1.5, 0) to (2,0);
\draw[->, shorten >=5pt,shorten <=5pt,out=30, in=150] (1.5, 0) to (2,0);
\draw[->, shorten >=5pt, shorten <=5pt,out=0, in=180] (1.5, 0) to (2,0);
\draw (1.75,.2) node{{\small $3$ }};
\draw (2.25,0) node{$F_{R,n_1+n_2+n_3}$};
\end{tikzpicture}.
\end{center}
and compare it with graph $G$ whose associated sequence is $(n_1,n_2-1,n_3,n_4+1,n_5)$, and decompose it as
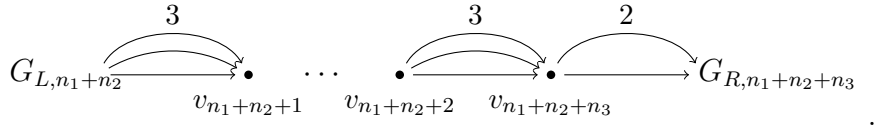
\begin{center}
\begin{tikzpicture}[scale=4]
\draw (-.1,0) node{$G_{L,n_1+n_2}$};
\draw[->, shorten >=5pt, shorten <=5pt,out=60, in=120] (0,0) to (.5,0);
\draw[->, shorten >=5pt,shorten <=5pt,out=30, in=150] (0,0) to (.5,0);
\draw[->, shorten >=5pt, shorten <=5pt,out=0, in=180]  (0,0) to (.5,0);
\draw (.25,.2) node{{\small $3$ }};
\fill (.5,0)  circle[radius=.4pt];
\draw (.5,-.1) node{ {\small $v_{n_1+n_2+1}$ }};
\draw (.75,0) node{$\cdots$};
\fill (1,0)  circle[radius=.4pt];
\draw (1,-.1) node{ {\small $v_{n_1+n_2+2}$ }};
\draw[->, shorten >=5pt, shorten <=5pt,out=60, in=120] (1, 0) to (1.5,0);
\draw[->, shorten >=5pt,shorten <=5pt,out=30, in=150] (1, 0) to (1.5,0);
\draw[->, shorten >=5pt, shorten <=5pt,out=0, in=180] (1, 0) to (1.5,0);
\draw (1.25,.2) node{{\small $3$ }};
\fill (1.5,0)  circle[radius=.4pt];
\draw (1.5,-.1) node{ {\small $v_{n_1+n_2+n_3}$ }};
\draw[->, shorten >=5pt, shorten <=5pt,out=60, in=120] (1.5, 0) to (2,0);
\draw[->, shorten >=5pt, shorten <=5pt,out=0, in=180] (1.5, 0) to (2,0);
\draw (1.75,.2) node{{\small $2$ }};
\draw (2.25,0) node{$G_{R,n_1+n_2+n_3}$};
\end{tikzpicture}.
\end{center}
Since $F_{L,n_1+n_2}$ and $G_{L,n_1+n_2}$ are isomorphic and $F_{R,n_1+n_2+n_3}$ and $G_{R,n_1+n_2+n_3}$ are isomorphic, they have the same number of paths, and we can easily compare:
\begin{equation*}
\begin{split}
\vert FP(F)\vert
&=\vert FP(F_{L,n_1+n_2}) \vert +2L_{n_1+n_2} \sum\limits_{k=0}^{n_3-1}3^k +\sum\limits_{k=1}^{n_3} 3^k R_{n_1+n_2+n_3} \\
&+2\sum\limits_{k=0}^{n_3} 3^k L_{n_1+n_2}R_{n_1+n_2+n_3} +\vert FP(F_{L,n_1+n_2+n_3}) \vert+\vert FP(F_{n_1+n_2+1,n_1+n_2+n_3}) \vert \\
\vert FP(G)\vert
&=\vert FP(G_{L,n_1+n_2}) \vert +L_{n_1+n_2} \sum\limits_{k=1}^{n_3}3^k +2\sum\limits_{k=0}^{n_3-1} 3^k R_{n_1+n_2+n_3} \\
&+2\sum\limits_{k=0}^{n_3} 3^k L_{n_1+n_2}R_{n_1+n_2+n_3} +\vert FP(G_{L,n_1+n_2+n_3}) \vert+\vert FP(G_{n_1+n_2+1,n_1+n_2+n_3}) \vert 
\end{split}
\end{equation*}
Because $F_{L,n_1+n_2}$ has sequence $(n_1,n_2-1,0,0,0)$ and $F_{L,n_1+n_2+n_3}$ has sequence $(0,0,0,n_4,n_5)$ then
\begin{equation*}
\begin{split}
L_{n_1+n_2}
&=2^{n_2-1}(2+n_1)-1
\ge 2^{n_4+3}(2+n_1)-1
2^{n_4} \cdot 8(2+n_1)-1 \\
&\ge 2^{n_4} \frac{2+n_5}{2+n_1}(2+n_1)-1
=R_{n_1+n_2+n_3}
\end{split}
\end{equation*}
where the above inequality follows from the lemma \ref{lem3.1} because $\frac{2+n_5}{2+n_1}\le \frac{5}{3}\le 8$
\begin{equation*}
\begin{split}
\vert FP(G)\vert-\vert FP(F)\vert
&=L_{n_1+n_2}\sum\limits_{k=0}^{n_3-1}3^k-R_{n_1+n_2+n_3}\sum\limits_{k=0}^{n_3-1} 3^k
=\left( L_{n_1+n_2}-R_{n_1+n_2+n_3} \right)\sum\limits_{k=0}^{n_3-1} 3^k > 0.
\end{split}
\end{equation*}
This contradicts the optimality of $F$.
The proof for the case $n_4-n_2\le 1$ is dual.
\end{proof}
\begin{lemma}\label{extremedies}
Let $F\in \mathcal{T}_{3,1}^N$ be such that $\sup_{E\in\mathcal{T}^N_{3,1}}P^N\!(E)=P^N\!(F)$. Then we have
\begin{equation*}
	|s_F^{-1}(v_i)|=1 \iff \; i\in \{1,n-1\},
\end{equation*}
where $F^0=\{v_1,\dots,v_n\}$ and $N$ is sufficiently large (at least $18$).
\end{lemma}
\begin{proof}
	We have already proven the implication $\impliedby$ -- this is exactly Lemma~\ref{lem3}. We shall therefore focus solely on the implication $\implies$. In light of Corollary~\ref{edgeseq} we only need to argue that we can have only one $1$ at the start of the sequence $l=(|s_F(v_i)^{-1}|)_{i=1}^{n-1}$ and a dual argument would yield the same for the ending of this sequence. Let us denote the sequence $l$ as in Corollary~\ref{edgeseq}:
	\begin{equation}
		l =: (n_1,n_2,n_3,n_4,n_5).
	\end{equation}
	Suppose that $n_1>1$. We need to consider four cases, based on the numbers $n_2,n_4$ and $n_3$ later on. If $n_2>1$, then consider the following modification: Let $G=(G_0,G_1,s_G,t_G)\in\mathcal{T}_{3,1}^N$ be the trunk with $G_0=\{v_2,\dots,v_n\}$ and the following sequence $l'=(|s_G^{-1}(v_i)|)_{i=2}^{n-1}$
	\begin{equation}
		l '=: (n_1-1,n_2-1,n_3+1,n_4,n_5).
	\end{equation}
	Graphically speaking, we construct $G$ by deleting the leftmost vertex of $F$ and putting the stray edge on top of the last $2$ before the sequence of $3$'s. In doing so, we lost exactly all paths that started with the very first vertex in $F$ and gained all paths in $G$ that factor through the newly added edge. Let us denote all paths in $F$ that start with the first vertex by $P_1$ and paths in $G$ that go through the newly added edge by $P_2$. Lost paths are fairly easy to compute:
	\begin{gather}
		P_1 = (\sum_{k=0}^{n_1}1) + (\sum_{k=1}^{n_2}2^k) + 2^{n_2}(\sum_{k=1}^{n_3}3^k) + 2^{n_2}3^{n_3}(\sum_{k=1}^{n_4}2^k) + 2^{n_2+n_4}3^{n_3}(\sum_{k=1}^{n_5}1) = \nonumber\\
		= n_1 - 1  + 2^{n_2}\left( 2 +  (\sum_{k=1}^{n_3}3^k) + 3^{n_3}(\sum_{k=1}^{n_4}2^k) + 2^{n_4}3^{n_3}n_5\right)
	\end{gather}
	As for the gained paths, the newly added edge turned the last $2$ in the sequence $l$ into a $3$. The the last $2$ in the sequence $l$ is in the ($n_1+n_2$)-th spot, i.e. the last two vertices that are connected by a pair of edges are vertices $n_1+n_2$ and $1+n_1+n_2$. We are interested in counting paths that go through these edges but go through only one particular edge between them. In other words, we look at the following graph:
	\begin{equation*}
		G_{L,n_1+n_2}\xrightarrow{2} v_{n_1+n_2} \to v_{1+n_1+n_2} \xrightarrow{3} G_{R,1+n_1+n_2}
	\end{equation*}
	The number of paths that go through the edge in the middle is given by
	\begin{equation}
		P_2=L_{1+n_1+n_2}R_{2+n_1+n_2}.
	\end{equation}
	Recall that $L_{1+n_1+n_2}$ is the number of paths that end at the unique sink of the following graph
	\begin{equation*}
		v_2\to \dots\to v_{n_1+1}\xrightarrow{2}v_{n_1+2}\xrightarrow{2}\dots\xrightarrow{2} v_{n_1+n_2}.
	\end{equation*}
	Thus $L_{1+n_1+n_2}$ is given by
	\begin{equation}
		L_{1+n_1+n_2} = (\sum_{k=0}^{n_2-1}2^k) + 2^{n_2-1}(n_1-1) = 2^{n_2}-1 + 2^{n_2-1}(n_1-1).
	\end{equation}
	As for $R_{1+n_1+n_2}$ we need to compute the number of paths that start at the unique source of the following graph
	\begin{equation*}
		v_{1+n_1+n_2}\xrightarrow{3}\dots\xrightarrow{3}v_{1+n_1+n_2+n_3}\xrightarrow{2}\dots\xrightarrow{2}v_{1+n_1+n_2+n_3+n_4}\to\dots\to v_n.
	\end{equation*}
	So it is given by
	\begin{equation}
		R_{1+n_1+n_2} = 1+ \sum_{k=1}^{n_3}3^k + 3^{n_3}(\sum_{k=1}^{n_4}2^k) + 3^{n_3}2^{n_4}n_5.
	\end{equation}
	So the difference $P_2-P_1$ can be simplified into
	\begin{equation}
		P_2-P_1 = (2^{n_2-1}(n_1-1)-1)R_{1+n_1+n_2} - n_1 + 1 - 2^{n_2}.
	\end{equation}
	Recall that we have no assumption on $n_3\ge0$, but we assume that $n_1>1,n_5>0, n_2>1$ and additionally $n_4\ge n_2-1$ by Lemma~\ref{lem5}. Hence we have the following bound on $R_{1+n_1+n_2}$:
	\begin{equation}
		R_{1+n_1+n_2} > 1 + 2(2^{n_4-1}-1)+2^{n_4}\ge 3\cdot2^{n_2-1}-1
	\end{equation}
	This lets us get the following bound for $P_2-P_1$:
	\begin{gather}
		P_2-P_1 > (3\cdot 2^{n_2-1}-1)(2^{n_2-1}(n_1-1)-1)-n_1+1-2^{n_2}\ge\nonumber\\
		\ge 3\cdot 2^{2n_2-2}-6\cdot 2^{n_2-1} = 2^{n_2-1}(3\cdot 2^{n_2-1}-6)\ge0
	\end{gather}
	Therefore our modified graph $G$  has a larger total amount of paths in comparison to the optimal graph $F$ which is obviously not possible.
	
	Next, let us consider the case $n_2=1$ and $n_4>1$. Here we do an analogous construction, just on the other set of twos: Consider the graph $E=(E^0,E^1,s_E,t_E)$ with $E^0=\{v_2,\dots,v_{n}\}$ and the following sequence of edges $l''=(|s_E^{-1}(v_i)|)_{i=2}^{n-1}$:
	\begin{equation}
		l'' = (n_1-1,1,n_3+1,n_4-1,n_5).
	\end{equation}
	The calculus of lost and gained paths is largely the same here. In fact we lose the exact same amount of paths, but we specialize to $n_2=1$, so we get
	\begin{equation}
		P_1 = n_1-1 + 2\left( 2+(\sum_{k=1}^{n_3}3^k) + 3^{n_3}(\sum_{k=1}^{n_4}2^k) + 2^{n_4}3^{n_3}n_5 \right).
	\end{equation}
	In this case we gain paths that go through a distinguished edge between the $(1+n_1+n_2+n_3)$-th vertex and its successor. So what we are after is $P_2'=L_{1+n_1+n_2+n_3}R_{2+n_1+n_2+n_3}$, where $L_{1+n_1+n_2+n_3}$ is the number of paths that end at the unique sink in the following graph
	\begin{equation*}
		v_2\to\dots\to v_{1+n_1}\xrightarrow{2}v_{1+n_1+n_2} \xrightarrow{3}\dots\xrightarrow{3}v_{1+n_1+n_2+n_3}
	\end{equation*}
	and $R_{2+n_1+n_2+n_3}$ is the number of paths that start at the unique source of the following graph
	\begin{equation*}
		v_{2+n_1+n_2+n_3}\xrightarrow{2}\dots\xrightarrow{2}v_{1+n_1+n_2+n_3+n_4}\to\dots\to v_n.
	\end{equation*}
	These numbers turn out to be
	\begin{align*}
		& L_{1+n_1+n_2+n_3} = (\sum_{k=0}^{n_3}3^k) + 2\cdot 3^{n_3}n_1 =  (\sum_{k=0}^{n_3}3^k) + 2\cdot 3^{n_3}(n_1-1)+2\cdot3^{n_3},\\
		& R_{2+n_1+n_2+n_3} = (\sum_{k=0}^{n_4-1}2^k) + 2^{n_4-1} n_5 .
	\end{align*}
	So the difference $P_2'-P_1$ simplifies to
	\begin{gather*}
		P_2'-P_1 = \left((\sum_{k=0}^{n_3}3^k)+2\cdot3^{n_3}(n_1-2)\right)\left((\sum_{k=0}^{n_4-1}2^k)+2^{n_4-1}n_5\right) + 2\left(3^{n_3}(\sum_{k=2}^{n_4}2^{k})+2^{n_4}3^{n_3}n_5\right) - P_1 =\\
		= \left((\sum_{k=0}^{n_3}3^k)+2\cdot3^{n_3}(n_1-2)\right)\left((\sum_{k=0}^{n_4-1}2^k)+2^{n_4-1}n_5\right) - n_1 -3 - 2(\sum_{k=1}^{n_3}3^k) = \\
		= \left((\sum_{k=0}^{n_3}3^k)+2\cdot3^{n_3}(n_1-2)\right)\left((\sum_{k=0}^{n_4-1}2^k)+2^{n_4-1}n_5-3\right) + (\sum_{k=0}^{n_3}3^k)+6\cdot3^{n_3}(n_1-2) - n_1-3\ge \\
		\ge (1+2(n_1-2))2 + 1 + 6(n_1-2) - n_1-3 = 9n_1 -16>0
	\end{gather*}
	Again, our modified graph $E$ turns out to be better. 
	
	Now consider the case $n_2=n_4=1$ and $n_3>0$. Here we can consider the following graph: Let $H=(H^0,H^1,s_H,t_H)$ be a graph  with $H^0=F^0$ the following sequence of edges $l'''=(|s_F^{-1}(v_i)|)_{i=1}^{n-1}$:
	\begin{equation*}
		l'''=(n_1-1,2,n_3-1,2,n_5).
	\end{equation*}
	This modification is very different from the previous ones -- this time we end up with the same number of vertices instead of lessening the total amount of vertices that we got previously. To compare the number of paths between the graphs $F$ and $H$ note that $H$ differs from $F$ by redistributing one of the $3$'s that lie next to a $2$ and redistributing one of its edges into a $1$ on the other side. This means that through our construction we lose paths in $F$ that go through that particular edge that we redistribute, we will denote the number of such paths by $Q_1$. On the flip side, we gain paths in $H$ that go through this redistributed path, let us denote the number of such paths by $Q_2$. Explicitly, $Q_1$ is given by the number of paths that pass through the middle edge in the following graph
	\begin{equation*}
		F_{L,1+n_1+n_2+n_3}\xrightarrow{3}v_{n_1+n_2+n_3}\to v_{1+n_1+n_2+n_3}\xrightarrow{2}F_{R,2+n_1+n_2+n_3}
	\end{equation*}
	So $Q_1=L_{1+n_1+n_2+n_3}R_{2+n_1+n_2+n_3}$. To compute $L_{1+n_1+n_2+n_3}$ we need to compute the number of paths that end at the unique sink of the following graph
	\begin{equation*}
		v_1\to\dots\to v_{1+n_1}\xrightarrow{2}v_{1+n_1+n_2}\xrightarrow{3}\dots\xrightarrow{3}v_{n_1+n_2+n_3}.
	\end{equation*}
	This yields the following number
	\begin{equation}
		L_{1+n_1+n_2+n_3} = (\sum_{k=0}^{n_3-1}3^k) + 2\cdot 3^{n_3-1}(n_1+1) = \frac{7}{2}3^{n_3-1}+2\cdot 3^{n_3-1}n_1 - \frac{1}{2}.
	\end{equation}
	To find $R_{2+n_1+n_2+n_3}$ we look at paths that start at the unique source of the following graph
	\begin{equation*}
		v_{1+n_1+n_2+n_3}\xrightarrow{2}v_{1+n_1+n_2+n_3+n_4}\to\dots\to v_n.
	\end{equation*}
	We obtain
	\begin{equation}
		R_{2+n_1+n_2+n_3} = 1 + 2(n_5+1) = 2n_5+3.
	\end{equation}
	The number $Q_2$ of paths gained is given by the product $Q_2=L_{1+n_1}R_{2+n_1}$. The number on the left is simply $L_{1+n_1}=n_1$. The other number is given by paths that start at the unique source of the following graph
	\begin{equation*}
		v_{1+n_1}\xrightarrow{2}v_{1+n_1+n_2}\xrightarrow{3}\dots\xrightarrow{3}v_{n_1+n_2+n_3}\xrightarrow{2}v_{1+n_1+n_2+n_3}\xrightarrow{2}v_{1+n_1+n_2+n_3+n_4}\to\dots\to v_n.
	\end{equation*}
	So we get
	\begin{gather}
		R_{2+n_1} = 1 + 2((\sum_{k=0}^{n_3-1}3^k) + 2\cdot3^{n_3-1}(1+2(n_5+1))) = \nonumber\\
		= 3^{n_3} + 4\cdot 3^{n_3-1} (3+2n_5) = 5\cdot3^{n_3} + 8\cdot 3^{n_3-1}n_5
	\end{gather}
	As the problem is symmetric in variables $n_1$ and $n_5$ we can assume $n_1\ge n_5\ge 1$. That way we get obtain the following estimate on $Q_2-Q_1$:
	\begin{gather}
		Q_2-Q_1 = n_1(5\cdot 3^{n_3}+8\cdot 3^{n_3-1}) - (2n_5+3)\left(\frac{7}{2}3^{n_3-1}+2\cdot 3^{n_3-1}n_1 - \frac{1}{2}\right) = \nonumber\\
		= 4\cdot 3^{n_3-1}n_1n_5 + 3^{n_3+1}n_1+(1-7\cdot 3^{n_3-1})n_5 + \frac{3}{2} - \frac{7}{2}3^{n_3} \ge \nonumber\\
		\ge 12\cdot3^{n_3-1} + \frac{7}{2} - \frac{7}{2}\cdot 3^{n_3}>0.
	\end{gather}
	Thus our modification would have increased the total amount of paths, which is a contradiction. Of course, this construction does not make sense if $n_3>0$. 
	
	Our final case would be if $n_2=n_4=1$ and $n_3=0$. Here we will consider the following modification -- let $K=(K^0,K^1,s_K,t_K)$ be the graph with $K^0=\{v_2,\dots,v_n\}$ and the following sequence $l_K=(|s_K^{-1}(v_i)|)_{i=2}^{n-1}$:
	\begin{equation}
		l_K = (n_1-1,2,0,1,n_5).
	\end{equation}
	To compare the number of paths between $F$ and $K$, we shall note that by doing this modification we lose paths going through the first vertex in $F$, but we get the paths that go through a fixed edge in the first $2$ that appears in $K$. The first half we have already computed -- this is $P_1$ from the start of this proof, but specialized to our parameters:
	\begin{equation}
		P_1 = n_1+4n_5+7.
	\end{equation}
	The other part, which we denote by $P_K$, is the number of paths that pass through the edge in the middle of the following graph:
	\begin{equation*}
		K_{L,n_1}\to v_{n_1}\to v_{1+n_1}\xrightarrow{2}K_{R,2+n_1}
	\end{equation*}
	Hence $P_K=L_{n_1}R_{2+n_1}$ and as always. In this case these numbers are fairly easy to compute:
	\begin{align*}
		& L_{n_1} = n_1-1,\\
		& R_{2+n_1} = 1+2+4(n_5+1) = 4n_5 + 7.
	\end{align*}
	From symmetry of this case again we can assume that $n_1\ge n_5\ge1$ without loss of generality. Thus the difference is given by:
	\begin{gather}
		(n_1-1)(4n_5+7) - n_1-4n_5-7 = 4n_1n_5 +6 n_1-8n_5 - 14\ge 2n_1-14.
	\end{gather}
	This means that if $n_1>7$, then our modification would increase the total amount of paths. As $n_1\ge n_5$, then that means would not work only for all graphs with at most seven $1$'s and two $2$'s. Such graphs have at most $18$ edges and so all of these edge cases are covered in our data (there are examples with multiple $1$'s on the edges).
\end{proof}

We are now ready  to prove our main result:
\begin{theorem}\label{main}
Let $F\in \mathcal{T}^N_3$. If $\sup_{E\in\mathcal{T}^N_3}P^N\!(E)=P^N\!(F)$ and $(1,n_2,n_3,n_4,1)$ be its associated sequence and $n_3>1$ then $\max \{n_2,n_4\}\le 2$.
\end{theorem}
\begin{proof}
Let us assume $n_2\ge 3$ and decompose $F$ as follows:
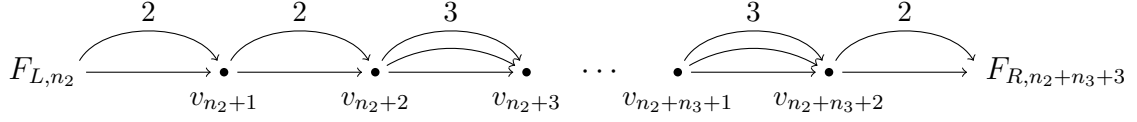
\begin{center}
\begin{tikzpicture}[scale=4]
\draw (-.1,0) node{$F_{L,n_2}$};
\draw[->, shorten >=5pt, shorten <=5pt,out=60, in=120] (0,0) to (.5,0);
\draw[->, shorten >=5pt, shorten <=5pt,out=0, in=180]  (0,0) to (.5,0);
\draw (.25,.2) node{{\small $2$ }};
\fill (.5,0)  circle[radius=.4pt];
\draw (.5,-.1) node{ {\small $v_{n_2+1}$ }};
\draw[->, shorten >=5pt, shorten <=5pt,out=60, in=120] (.5,0) to (1,0);
\draw[->, shorten >=5pt, shorten <=5pt,out=0, in=180]  (.5,0) to (1,0);
\draw (.75,.2) node{{\small $2$ }};
\fill (1,0)  circle[radius=.4pt];
\draw (1,-.1) node{ {\small $v_{n_2+2}$ }};
\draw[->, shorten >=5pt, shorten <=5pt,out=60, in=120] (1, 0) to (1.5,0);
\draw[->, shorten >=5pt,shorten <=5pt,out=30, in=150] (1, 0) to (1.5,0);
\draw[->, shorten >=5pt, shorten <=5pt,out=0, in=180] (1, 0) to (1.5,0);
\draw (1.25,.2) node{{\small $3$ }};
\fill (1.5,0)  circle[radius=.4pt];
\draw (1.5,-.1) node{ {\small $v_{n_2+3}$ }};
\draw (1.75,0) node{$\cdots$};
\fill (2,0)  circle[radius=.4pt];
\draw (2,-.1) node{ {\small $v_{n_2+n_3+1}$ }};
\draw[->, shorten >=5pt, shorten <=5pt,out=60, in=120] (2, 0) to (2.5,0);
\draw[->, shorten >=5pt,shorten <=5pt,out=30, in=150] (2, 0) to (2.5,0);
\draw[->, shorten >=5pt, shorten <=5pt,out=0, in=180] (2, 0) to (2.5,0);
\draw (2.25,.2) node{{\small $3$ }};
\fill (2.5,0)  circle[radius=.4pt];
\draw (2.5,-.1) node{ {\small $v_{n_2+n_3+2}$ }};
\draw[->, shorten >=5pt, shorten <=5pt,out=60, in=120] (2.5, 0) to (3,0);
\draw[->, shorten >=5pt, shorten <=5pt,out=0, in=180] (2.5, 0) to (3,0);
\draw (2.75,.2) node{{\small $2$ }};
\draw (3.25,0) node{$F_{R,n_2+n_3+3}$};
\end{tikzpicture}.
\end{center}
Let us considet the graph $G$ whose associated sequence is $(1,n_2-2,n_3+2,n_4-1,1)$ and decompose is as
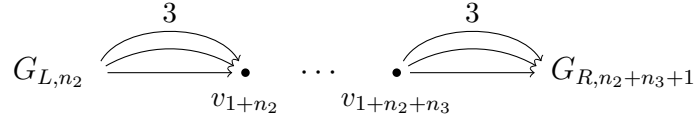
\begin{center}
\begin{tikzpicture}[scale=4]
\draw (-.15,0) node{$G_{L,n_2}$};
\draw[->, shorten >=5pt, shorten <=5pt,out=60, in=120] (0,0) to (.5,0);
\draw[->, shorten >=5pt,shorten <=5pt,out=30, in=150] (0,0) to (.5,0);
\draw[->, shorten >=5pt, shorten <=5pt,out=0, in=180]  (0,0) to (.5,0);
\draw (.25,.2) node{{\small $3$ }};
\fill (.5,0)  circle[radius=.4pt];
\draw (.5,-.1) node{ {\small $v_{1+n_2}$ }};
\draw (.75,0) node{$\cdots$};
\fill (1,0)  circle[radius=.4pt];
\draw (1,-.1) node{ {\small $v_{1+n_2+n_3}$ }};
\draw[->, shorten >=5pt, shorten <=5pt,out=60, in=120] (1, 0) to (1.5,0);
\draw[->, shorten >=5pt,shorten <=5pt,out=30, in=150] (1, 0) to (1.5,0);
\draw[->, shorten >=5pt, shorten <=5pt,out=0, in=180] (1, 0) to (1.5,0);
\draw (1.25,.2) node{{\small $3$ }};
\draw (1.75,0) node{$G_{R,n_2+n_3+1}$};
\end{tikzpicture}.
\end{center}
Since both $F_{L,n_2}$ and $G_{L,n_2}$ have associated sequence $(1,n_2-2,0,0,0)$ let us denote by $L$ the number of paths ending at their sinks.
Similarly both $F_{R,n_2+n_3+3}$ and $G_{R,n_2+n_3+1}$ have associated sequence $(0,0,0,n_4-1,1)$ and we denote by $R$ the number of paths begining at their sources.
\begin{equation*}
\begin{split}
\vert FP(F)\vert&= \vert FP(F_{L,n_2}) \vert +\vert FP(F_{n_2+1,n_2+n_3+2}) \vert +2L+4L\sum\limits_{k=0}^{n_3} 3^k \\
&+2R\sum\limits_{k=0}^{n_3} 3^k+4\cdot 3^{n_3}R+8\cdot 3^{n_3}RL +\vert FP(F_{R,n_2+n_3+3})\vert \\
\vert FP(G)\vert&=\vert FP(G_{L,n_2}) \vert +\vert FP(G_{1+n_2,1+n_2+n_3}) \vert +L\sum\limits_{k=1}^{n_3+1} 3^k \\
&+R\sum\limits_{k=1}^{n_3+1} +3^{n_3+2} LR +\vert FP(G_{R,n_2+n_3+1})\vert
\end{split}
\end{equation*}
Calculating the difference we get:
\begin{equation*}
\begin{split}
\vert FP(G)\vert-\vert FP(F)\vert
=3^{n_3}LR-3^{n_3+1}-\frac{3}{2}L(3^{n_3}+1)-R \left( \frac{5}{2}3^{n_3}+\frac{1}{2} \right)
\end{split}
\end{equation*}
We have to consioder two cases: \\
Case 1. 
$n_2=n_4$, then $L=R$ and we get:
\begin{equation*}
\begin{split}
\vert FP(G)\vert-\vert FP(F)\vert
=3^{n_3}L^2-(4\cdot 3^{n_3}+2)L-3^{n_3+1}
\end{split}
\end{equation*}
We get $\Delta=4(7\cdot 3^{2n_3}+4\cdot 3^{n_3}+1)$ and the bigger of the roots is given by:
\begin{equation*}
L
=\frac{4\cdot 3^{n_3}+2+2\sqrt{7\cdot 3^{2n_3}+4\cdot 3^{n_3}+1}}{2\cdot 3^{n_3}}
=2+\sqrt{7+4\cdot 3^{-n_3}+3^{-2n_3}} +3^{-n_3}.
\end{equation*}
Since as a function of $n_3$ it is a decreasing function and for $n_3=1$ it gives $L=\frac{7+2\sqrt{19}}{3}>5$ we see the swich does not increases the number of paths, but for $n_3=2$ it gives $L=\frac{19+2\sqrt{151}}{9}<5$ therefore the switch increases the number of paths. \\
Case 2.
$n_2=n_4+1$, then $L=2R+1$ and we get:
\begin{equation*}
\begin{split}
\vert FP(G)\vert-\vert FP(F)\vert
=2\cdot 3^{n_3}L^2-\left( \frac{11}{2}\cdot 3^{n_3}+2 \right)L-\frac{11}{2}3^{n_3}-\frac{1}{2}
\end{split}
\end{equation*}
We get $\Delta=\frac{1}{4}\left( 297\cdot 3^{2n_3}+104\cdot 3^{n_3}+16 \right)$ and the bigger of the roots is given by:
\begin{equation*}
L=\cfrac{\frac{11}{2} \cdot 3^{n_3}+2+\frac{1}{2}\sqrt{297\cdot 3^{2n_3}+104\cdot 3^{n_3}+16}}{4\cdot 3^{n_3}}=\frac{11}{8}+\frac{1}{8}\sqrt{297+104\cdot 3^{-n_3}+16 3^{-2n_3}} +\frac{1}{2}\cdot 3^{-n_3}
\end{equation*}
Since as a function of $n_3$ it is a decreasing function and for $n_3=1$ it gives $L=\frac{37+\sqrt{3001}}{24}<5$ we see that the switch increases the number of paths.
\end{proof}
 
\newpage
\subsection{Three Sisters}
\begin{lemma}\label{centrepaths}
	Let $F$ be a trunk graph of the form $(0,0,m,0,0)$, with $m\ge 1$. Then the number $|FP(F)|$ of finite paths in $F$ is given by
$$
|FP(F)|=\frac{1}{4}\left( 3^{m+2}-2(m+2)-1\right).
$$
\end{lemma}
\begin{proof}
	This is an easy inductive proof. Notice that if $F=(0,0,1,0,0)$, then the above formula clearly holds. Suppose that the above formula holds for some $m\ge 1$ and $F=(0,0,m+1,0,0)$. Let us denote by $v$ the last vertex of $F$. Observe that
	\begin{equation*}
		FP(F) = \{\text{Paths ending before }v\} \sqcup \{\text{Paths ending at }v\}.
	\end{equation*}
	The size of the first set is given by the inductive hypothesis as this is exactly the number of paths in a graph of the form $(0,0,m,0,0)$. The size of the second set is given by a geometric sum:
	\begin{equation*}
		|\{\text{Paths ending at }v\}| = 1 + 3 + 3^2+\dots + 3^{m+1} = \frac12(3^{m+2}-1).
	\end{equation*}
	Therefore we see that
	\begin{equation*}
		|FP(F)| = \frac14(3^{m+2}-2(m+2)-1) + \frac12(3^{m+2}-1) = \frac14(3^{m+3}-2(m+3)-1)
	\end{equation*}
	which finishes the proof.
\end{proof}
\begin{theorem}\label{3sis}
For $N \geq 16$, where $N$ is the number of edges, the following formulas compute $P^N_*$, 
the maximal number of paths for a graph with $N=3k+i$, where $i \in \{0,1,2 \}$:
\phantom{.}\\
\begin{center}
\setlength\extrarowheight{16pt}
\boxed{
\begin{tabular}{lll}
$i=0$&& $P^N_*=\displaystyle\frac{1}{4}\left( 121\cdot 3^{k-2}-2k+7\right)$\\
$i=1$&& $P^N_*=\displaystyle\frac{1}{4}\left( 529\cdot 3^{k-3}-2k+25\right)$\\
$i=2$ && $P^N_*=\displaystyle\frac{1}{4}\left( 253\cdot 3^{k-2}-2k+15\right)$\\
\end{tabular}}
\end{center}
\end{theorem}
\begin{proof}
	Consider the graph $F_1=(1,1,m,0,0)$. Let $v$ be the vertex in $F_1$ such that it emits three edges and its predecessor emits two edges. We want to find the number of paths in $F_1$ that start at $v$ or earlier and are not fully contained in the subgraph of $F_1$ of the form $(0,0,m,0,0)$. We will denote this number by $P_1(m)$. Now this number can be computed in an analogous way as we often did throughout this paper:
	\begin{equation*}
		P_1(m) = (1+3+\dots+3^{m})(2+2) + 3 = 2\cdot3^{m+1} + 1.
	\end{equation*}
	Likewise, let $F_2=(1,2,m,0,0)$, $v$ be the vertex that emits three edges and whose predecessor emits only two edges. Let $P_2(m)$ be the number of paths in $F_2$ that start at $v$ or earlier and are not fully contained in the $(0,0,m,0,0)$ part of our graph. Then $P_2(m)$ is given by
	\begin{equation*}
		P_2(m) = (1+3+\dots+3^{m})(2 + 4 + 4) + 8 = 5\cdot3^{m+1} + 3.
	\end{equation*}
	These formulas together with Lemma~\ref{centrepaths} will let us compute the formulas in the formulation of the theorem. Namely, observe that if $i=0$, then by Lemma~\ref{extremedies} and Theorem~\ref{main} our optimal graph has to be of the form $G_0=(1,1,k-2,1,1)$. Let $v$ be the vertex in $G_0$ that emits three edges and its predecessor emits two edges and $w$ be the vertex in $G_0$ that emits two edges and its predecessor emits three and $G_0^3=(0,0,k-2,0,0)$ be the full subgraph of $G_0$ that contains all vertices emitting three edges and $w$. We can decompose the set $FP(G_0)$ into the following subsets:
	\begin{gather}
		\nonumber FP(G_0) = \{\text{Paths that start at or before  }v\text{, not fully contained in }G_0^3\text{, ending before }w\}\sqcup\\  \label{decomposition}
		\sqcup\{\text{Paths that end at or after }w\text{, not fully contained in }G_0^3\text{, starting after }v\} \sqcup\\
		\nonumber \sqcup FP(G_0^3) \sqcup \{\text{Paths thru both }v, w\text{, not fully contained in }G_0^3\}.
	\end{gather}
	Now we need to find the cardinality of these sets.  Observe that the following holds:
	\begin{align*}
		& |\text{First summand in \eqref{decomposition}}| = P_1(k-3) = 2\cdot3^{k-2} + 1,\\
		& |\text{Second summand in \eqref{decomposition}}| = P_1(k-3) = 2\cdot3^{k-2}+1,\\
		& |\{\text{Paths in }G_0^3\}| = \frac14(3^k-2k-1),
	\end{align*}
	where the last formula is an application of Lemma~\ref{centrepaths}. As for the last remaining set, we see that
	\begin{align*}
		 |\{\text{Paths thru both }v, w\text{, not fully contained in }G_0^3\}| = 3^{k-2}(1+2+2)^2-3^{k-2}=24\cdot 3^{k-2}.\\
	\end{align*}
	In summary, we get
	\begin{equation*}
		P_*^N = 4\cdot 3^{k-2}+2 + \frac14(3^k - 2k - 1) + 24\cdot 3^{k-2} = \frac14(121\cdot3^{k-2} - 2k + 7).
	\end{equation*}
	The other cases are completely analogous. For the $i=1$ case, the optimal graph is given by $G_1=(1,2,k-3,2,1)$. We decompose $FP(G_1)$ in the same fashion as for $G_0$. Thus $|FP(G_1)|$ can be written as the following sum:
	\begin{equation*}
		|FP(G_1)| = P_2(k-4)+P_2(k-4) + \frac14(3^{k-1}-2(k-1)-1) + |\{\text{Paths thru both }v,w\}|.
	\end{equation*}
	The last summand is given by
	\begin{equation*}
		|\{\text{Paths thru both }v,w\}| = 3^{k-3}(1+2+4+4)^2-3^{k-3} = 120\cdot 3^{k-3}.
	\end{equation*}
	Thus our formula for $P_*^N$ becomes
	\begin{equation*}
		P_*^N = 10\cdot 3^{k-3}+6 + \frac14(3^{k-1}-2k+1) + 120\cdot 3^{k-3} = \frac14(529\cdot 3^{k-3}-2k+25).
	\end{equation*}
	Finally, for the $i=2$ case the optimal graph is of the form $G_2=(1,2,k-2,1,1)$ (equivalently its mirror image). The appropriate decomposition lets us describe $|FP(G_2)|$ as the following sum
	\begin{gather*}
		|FP(G_2)| = P_1(k-3)+P_2(k-3) + \frac14(3^k - 2k - 1) + 3^{k-2}(1+2+2)(1+2+4+4)-3^{k-2} = \\
		= 2\cdot 3^{k-2}+1 + 5\cdot 3^{k-2} + 3 + \frac14(3^k - 2k -1) + 54\cdot 3^{k-2} = \\
		= \frac14( 253\cdot 3^{k-2} - 2k + 15).
	\end{gather*}
	Which finishes the proof.
\end{proof}

\section{Leavitt path algebras}
\noindent
Every finite-dimensional Leavitt path algebra over complex numbers is a graph C*-algebra.

\begin{theorem}\cite[Proposition 3.5]{aapsm07}
Let $E$ be a finite acyclic graph with $N>0$ edges, and $\{v_i\}^n_1$ be the set of all its sinks. Denote by $p_i$ the number of all paths ending at~$v_i$. 
Then
$$
L_k(E)=\bigoplus_{i=1}^n M_{p_i}(k).
$$
\end{theorem}
\begin{remark} Out-splitting \cite{elersruiz} the  graph 
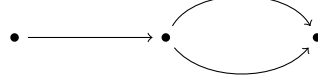
\begin{center}
\begin{tikzpicture}[scale=4]

\fill (0,0)  circle[radius=.4pt];
\draw[->, shorten >=5pt, shorten <=5pt] (0,0) to (.5,0);
\fill (.5,0)  circle[radius=.4pt];
\draw[->, shorten >=5pt, shorten <=5pt,out=60, in=120] (.5, 0) to (1,0);
\draw[->, shorten >=5pt, shorten <=5pt,out=-50, in=-130] (.5, 0) to (1,0);
\fill (1,0)  circle[radius=.4pt];

\end{tikzpicture}
\end{center}
at the middle vertex yields the graph 
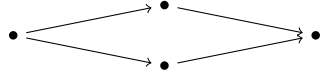
\begin{center}
\begin{tikzpicture}[scale=4]

\fill (0,0)  circle[radius=.4pt];
\draw[->, shorten >=5pt, shorten <=5pt] (0,0) to (.5,.10);
\draw[->, shorten >=5pt, shorten <=5pt] (0,0) to (.5,-.10);
\fill (.5,.10)  circle[radius=.4pt];
\fill (.5,.-.10)  circle[radius=.4pt];
\draw[->, shorten >=5pt, shorten <=5pt]  (.5,.10) to (1,0);
\draw[->, shorten >=5pt, shorten <=5pt]   (.5,.-.10) to (1,0);
\fill (1,0)  circle[radius=.4pt];

\end{tikzpicture}.
\end{center}
 The Leavitt path algebras of these graphs are graded isomorphic to $M_5(k)$.  
 This shows that the number of edges is not an invariant of graded Leavitt path algebras. 
\end{remark}

\subsection{From forests to trunks}
\begin{lemma}
Let $E$ be an $N$-forest graph. Then the maximal possible dimension of its Leavitt path algebra $L_k(E)$ is achieved for a graph with one sink.
\end{lemma}
\begin{proof}
For an N forest graph, let $p_k$ be the number of paths ending at a sink $v_k$, where $1 \leq k \leq n$. By Proposition 3.5 of \cite{aapsm07}, we know that the dimension of the associated Leavitt path algebra is $\sum_{k=1}^n p_k^2$. This is maximized when we identify all of the sinks; While we loose $n-1$ paths of length $0$ leading to the previously unique sinks, notice that 
\begin{align*}
 \left(\sum_{k=1}^n p_k -(n-1) \right)^2 &= \sum_{k=1}^n p_k^2 +(n-1)^2+\sum_{1 \leq i\neq k \leq n}2p_kp_i - \sum_{k=1}^n 2(n-1)p_k \\
 & = \sum_{k=1}^n p_k^2+(n-1)^2  +\sum_{k=1}^n \sum_{1 \leq i \neq k \leq n } 2p_k(p_i -n+1) \\ &= \sum_{k=1}^n p_k^2+(n-1)^2+\sum_{k=1}^n 2p_k(\sum_{1 \leq i \neq k \leq n } p_i -n+1)  \geq  \sum_{k=1}^n p_k^2
\end{align*} as $p_k \geq 2$ for  $1 \leq k \leq n$.
\end{proof}
Let $E$ be a connected acyclic graph with $N>0$ edges and one sink. We denote the set of all connected acyclic graph with $N>0$ edges and single sink by~$\mathcal{E}_s^N$, and we write $P_s^N\!(E)$ for the number of all paths in~$E$ ending in the sink.
Note that $\mathcal{T}^N \subset \mathcal{E}_s^N$.

\begin{lemma}\label{lem4.0}
The following equality holds:
$$
\sup_{E\in\mathcal{E}_s^N}P_s^N\!(E)=\sup_{E\in\mathcal{T}^N}P_s^N\!(E).
$$
\end{lemma}
\begin{proof}
Let $E\in\mathcal{E}_s^N$. Our proof consists of the following three constructions, all of which can only enlarge the value $P_s^N(E)$:
\begin{enumerate}
	\item \emph{Splitting sources}. For any source $v\in E^0$ such that $|s_E^{-1}(v)|\ge2$, consider the graph $F=(F^0,F^1,s_F,t_F)$, whose set of vertices is
	\begin{equation*}
		F^0 := \left(E^0\setminus\{v\}\right)\sqcup \{v_e:e\in s_E^{-1}(v)\}.
	\end{equation*}
	Both the set of edges $F^1:=E^1$ and the target map $t_F:=t_E$ we set to be the same. The source map $s_F$ is defined as follows
	\begin{equation*}
		s_F(e) = \begin{cases}
			v_e & \text{ if }e\in s_E^{-1}(v),\\
			s_E(e) & \text{ otherwise.}
		\end{cases}
	\end{equation*}
	Clearly $F\in\mathcal{E}^N$. Note that $P_s^N(F)= P_s^N(E)$ as $F$ preserves the amount of paths of positive length. We say that the graph $F$ came from the graph $E$ by splitting of the source $v$. 
	\item \emph{Combing stray paths}. Suppose that inside $E$ we have two paths $e=(e_1,\dots,e_n)$ and $v=(v_1,\dots,v_k)$ such that there exists $i\in\{1,\dots,n\}$ for which we have $t(v_k)=t(e_i)$ and $s_E(v_j)$ is never a vertex that $e$ runs through. We can define a new graph $G=(G^0,G^1,s_G,t_G)$, where the set of vertices, the set of edges and the source map stays the same $G^0:=E^0, G^1:=E^1, s_G:=s_E$. The target map $t_G$ differs from $t_E$ only on $v_k$: 
	\begin{equation*}
		t_G(e) = \begin{cases}
			s_E(e_1) & \text{ if }e=v_k,\\
			t_E(e) & \text{ otherwise}.
		\end{cases}
	\end{equation*} 
	Clearly $G$ is connected has the same amount of edges as $E$ and our assumption of all $s_E(v_j)$ never being a vertex that $e$ runs through ensures that $G$ remains acyclic. Thus $G\in\mathcal{E}^N$ and note that $P_s^N(G)= P_s^N(E)$ as we can define an bijection (restricted to paths that end at the sink) from  $FP(E)\to FP(G)$ that sends any path $w$ ending at the sink that does not contain $v_k$ to itself -- note that our graphs are completely the same if we cut out this particular edge, so this is a well-defined mapping. If $w$ contains $v_k$, say $w=(w_1,\dots,w_m)$ with $v_k=w_j$ for some $j\in\{1,\dots,m\}$, then we shall send $w$ to the following path
	\begin{equation*}
		w'=(w_1,\dots,w_j,e_1,\dots,e_i,w_{j+1},\dots,w_m).
	\end{equation*}
	In other words, we put $v_k,e_1,\dots,e_i$ instead of $v_k$ in $w$. This map is bijective as we only consider paths that end at the sink. Similarly as in the first procedure we can dualize this construction to paths that start on the fixed path $(e_1,\dots,e_n)$. Of course this dual construction also preserves the total amount of paths.
	\item \emph{Resolving alternative paths}. Suppose that we two paths $e=(e_1,\dots,e_n),v=(v_1,\dots,v_m)$, with $n\ge m>1$, such that $s_E(e_1)=s_E(v_1)$ and $t_E(e_n)=t_E(v_m)$, and there is no path connecting vertices that $e$ and $v$ run through. We want to merge vertices that $v$ runs through with vertices that $e$ runs through -- this operation reduces the total amount of vertices, but the total amount of paths increases, and thus it increases the muber of paths that end at the sink. Explicitly, we construct a graph $H=(H^0,H^1,s_H,t_H)$ as follows: Set the vertices to be
	\begin{equation*}
		H^0 := E^0 \setminus \{ t(v_k):k=1,\dots,m-1 \}
	\end{equation*}
	and the edges remain the same $H^1:=E^1$. We change the source and target maps to reflect the process of merging together our desired vertices
	\begin{align*}
		& s_H(e) := \begin{cases}
			s_E(e_i) & \text{ if } e=v_i \text{ for some }i\in\{ 1,\dots,m \},\\
			s_E(e) & \text{otherwise},
		\end{cases}\\
		& t_H(e) := \begin{cases}
			t_E(e_i) & \text{ if } e=v_i \text{ for some }i\in\{ 1,\dots,m \},\\
			t_E(e) & \text{otherwise}.
		\end{cases}
	\end{align*}
	Our final graph $H$ clearly is connected and has the same amount of edges that $E$ does and our assumption of having no paths connecting that $v$ and $e$ run through we get that $H$ must still be acyclic. Thus $H\in\mathcal{E}^N$. In order to see that $P_s^N(H)>P_s^N(E)$ observe that any path of positive length $w$ to the sink in $E$ is still a valid path in $H$, but we get more -- in our new graph $H$ we can mix the $v_i$'s with $e_j$'s, hence we get at least $2^m-2$ new paths.
\end{enumerate}
Having all of these moves defined we modify our graph $E$ in the following way. First, we split all sources and let $F$ be the final product of all of these 
procedures. The graph $F$ will be connected as it has a unique sink. Next, we can identify all instances of a path splitting into multiple paths and merges together later on. Suppose that we have two paths $e=(e_1,\dots,e_n)$ and $v=(v_1,\dots,v_m)$. It might be the case that there exists a path connecting two vertices that $v$ and $e$ runs through, say $w$ is a path such that $s(w)=s_E(v_i)$ and $t(w)=t_E(e_j)$ for some $i=1,\dots,m-1$ and $j=2,m$. Therefore we can first consider the splitting of paths that certainly occurs at $s_E(v_i)$, since $(w,e_{j+1},\dots,e_n)$ (or in the case $j=m$, just $w$ by itself) and $(v_i,\dots,v_m)$ are two paths with the same source and target. Note that if we resolve the alternative routes which are $(w,e_{j+1},\dots,e_n)$ and $(v_i,\dots,v_m)$, then the paths $(v_1,\dots,v_{i-1})$ and $e_1,\dots,e_j$ form a new splitting of paths, provided that $i>1$, which we can take care of afterwards. We can therefore iteratively resolve all alternative path obtaining a new graph that we shall denote by $G$. The graph $G$ is now just a collection of paths that cross each other in a way that is acyclic. We can therefore pick any path in $G$ and comb all stray paths that are either going into and outgoing from our chosen graph. Thus the end product is a graph that lies in $\mathcal{T}^N$ and since all of our procedures only increase the total amount of paths, then our final product also has a larger total amount of paths. Thus for any graph $E_s\in\mathcal{E}^N$ we can find a better graph that lies in $\mathcal{T}^N$, which finishes the proof.
\end{proof}

\newpage
\subsection{Experimental data} 
Here we again use Python to check the first 25 cases.
\phantom{.}\\

\begin{center}
\begin{tabular}{lllll}
\rowcolor{Gray}
$N$ edges         & Maximal number of paths   & Optimal trunks  &  Length  of longest path \\
\rowcolor{LGray}
3&5&(1,2)&2\\

4&7&(1,1,2),(1,3),(2,2)&2 or 3\\ \hline \hline
\rowcolor{Gray}
5&11&(1,2,2)&3\\ 
\rowcolor{LGray}
6&16&(1,2,3)&3\\

7&23&(1,2,2,2)&4\\
\rowcolor{Gray}
8&34&(1,2,2,3)&4\\
\rowcolor{LGray}
9&49&(1,2,3,3)& 4\\

10&70&(1,2,2,2,3)&5\\
\rowcolor{Gray}
11&103&(1,2,2,3,3)&5\\
\rowcolor{LGray}
12&148&(1,2,3,3,3)&5\\

13&211&(1,2,2,2,3,3)&6\\
\rowcolor{Gray}
14&310&(1,2,2,3,3,3)&6\\
\rowcolor{LGray} 
15&445&(1,2,3,3,3,3)&6 \\ 
16&634&(1,2,2,2,3,3,3)&7\\
\rowcolor{Gray}
17&931&(1,2,2,3,3,3,3)&7\\
\rowcolor{LGray}
18&1336&(1,2,3,3,3,3,3)&7\\

19&1903&(1,2,2,2,3,3,3,3)&8\\
\rowcolor{Gray}
20&2794&(1,2,2,3,3,3,3,3)&8\\
\rowcolor{LGray}
21&4009&(1,2,3,3,3,3,3,3)&8\\

22&5710&(1,2,2,2,3,3,3,3,3)&9\\
\rowcolor{Gray}
23&8383&(1,2,2,3,3,3,3,3,3)&9\\
\rowcolor{LGray}
24&12028&(1,2,3,3,3,3,3,3,3)&9\\
25&17131&(1,2,2,2,3,3,3,3,3,3)&10\\

\end{tabular}
\end{center}

\subsection{Main result}
\begin{lemma}\label{lem4.05}
Let $E\in\mathcal{T}^N$ with vertixes $\{v_1,\dots ,v_n\}$, then
\begin{equation}
P_s^N(E)=\sum\limits_{k=1}^{n-1} \prod\limits_{m=k}^{n-1} \vert s^{-1}(v_m) \vert+1
\end{equation}
\end{lemma}
\begin{proof}
Since for $E$ there is only one sink $v_n$ the number $P_s^N(E)$ is equal to the number of paths ending in vertex $v_n$.
The number of paths starting at $v_k$ where $k<n$ and ending in $v_n$ is equal to $\prod\limits_{m=k}^{n-1} \vert s^{-1}(v_m) \vert$.
This together with only path beggining and ending in $v_n$ gives the stated formula.
\end{proof}
\begin{lemma}\label{lem4.1}
The following equality holds:
$$
\sup_{E\in\mathcal{T}^N}P_s^N\!(E)=\sup_{E\in\mathcal{T}^N_3}P_s^N\!(E).
$$
\end{lemma}
\begin{proof}
Let $E\in\mathcal{T}^N$ and suppose that we have an edge $e\in E^1$ which has a label $n>3$ -- here we consider $E$ as a simple and labeled graph with labels coming from $\mathbb{N}$. We can blow such an edge up into two edges $e_1,e_2$ such that the label of $e^1$ is $\lfloor n/2\rfloor$ and the label of $e^2$ is $\lceil n/2 \rceil$. Explicitly, we define a graph $F=(F^0,F^1,s_F,t_F)$ as follows: Let the set of vertices and edges to be
\begin{align*}
	F^0 := E^0\sqcup\{\epsilon\}, && F^1 = \left(E^1\setminus \{e\}\right)\sqcup \{e^1,e^2\}.
\end{align*}
Next, we need to define the source and target maps:
\begin{align*}
	& s_F(w) := \begin{cases}
		s_E(e) & \text{ if }w=e^1,\\
		\epsilon & \text{ if }w=e^2,\\
		s_E(w) & \text{ otherwise}.
	\end{cases}\\
	& t_F(w) := \begin{cases}
		\epsilon & \text{ if }w=e^1,\\
		t_E(e) & \text{ if }w=e^2,\\
		t_E(w) & \text{ otherwise}.
	\end{cases}\
\end{align*}
Lastly, we need the labeling $l_F:F^1\to\mathbb{N}$ and it is defined as follows
\begin{equation*}
	l_F(w) := \begin{cases}
		\left\lfloor \frac{n}{2}\right\rfloor &\text{ if }w=e^1,\\
		\left\lceil \frac{n}{2} \right\rceil & \text{ if }w=e^2,\\
		l_E(w) &\text{ otherwise}.
	\end{cases}
\end{equation*}
Clearly any path in $E$ that does not contain $e$ corresponds to a unique path in $F$ that does not contain either $e^1$ or $e^2$. Let $L$ be the amount of paths 
in $E$ that end at $s_E(e)$ and $R_s$ be the amount of paths in $E$ that start at $t_E(e)$ that end at the sink. Then the total amount of paths that contain $e$ ending at the sink is 
\begin{equation*}
	T_E = L\cdot n\cdot R_s.
\end{equation*}
On the other hand, the total amount of paths in $F$ that end at the sink that contain either $e^1$ or $e^2$ ending at the sink is
\begin{equation*}
	T_F = L\cdot \left\lfloor \frac{n}{2}\right\rfloor\left\lceil \frac{n}{2} \right\rceil \cdot R_s + \left\lceil \frac{n}{2} \right\rceil \cdot R_s.
\end{equation*}
As $n\ge4$, it is clear that the first term of $T_F$ already overpowers $T_E$, thus our construction can only increase the total number of paths. Iterating this for all 
edges with label at least $4$ we get that for any graph $E\in\mathcal{T}^N$ we can find a better candidate that lies in $\mathcal{T}_3^N$, which finishes the proof.
\end{proof}

\begin{lemma}\label{lem4.1}
Let $F\in \mathcal{T}^N_3$. If $\sup_{E\in\mathcal{T}^N_3}P_s^N\!(E)=P_s^N\!(F)$, then then
\begin{displaymath}
\forall_{1\le i<j\le n-1} \ \vert s^{-1}(v_i) \vert \le \vert s^{-1}(v_j) \vert
\end{displaymath}
where $F^0=\{v_1,\dots,v_n\}$.
\end{lemma}

\begin{proof}
Let us fix $i, j \in \{1,\dots ,n-1\}$ where $i <j $.  The number of paths that end at the sink $v_n$ is:
\[P_s^N(F)=\sum_{k=1}^{n-1} \prod_{m=k}^{n-1} |s^{-1}(v_m)| +1. \]
which is:

\begin{displaymath}
\begin{split}
P_s^N(F)&=
\sum_{k=1}^{i} \prod_{m=k}^{n-1} |s^{-1}(v_m)| +  \sum_{k=i+1}^{j} \prod_{m=k}^{n-1} |s^{-1}(v_m)| +\sum_{k=j+1}^{n-1} \prod_{m=k}^{n-1} |s^{-1}(v_m)| +1\\
&=\vert s^{-1}(v_i)\vert \vert s^{-1}(v_j) \vert \sum_{k=1}^{i} \prod_{\substack{m=k \\ m\not=i \\ m\not =j}}^{n-1} |s^{-1}(v_m)| +\vert s^{-1}(v_j)\vert \sum_{k=i+1}^{j} \prod_{\substack{m=k \\ m\not =j}}^{n-1} |s^{-1}(v_m)|+\sum_{k=j+1}^{n-1} \prod_{m=k}^{n-1} |s^{-1}(v_m)| +1
\end{split}
\end{displaymath}

Let us look at the trunk where we switch $v_i$ and $v_j$. 
Thus, we go from 
\[F=(|s^{-1}(v_1)|, |s^{-1}(v_2)|, \ldots , |s^{-1}(v_i)|, \ldots |s^{-1}(v_j)|, \ldots, |s^{-1}(v_n)|)\] 
to 
\[G=(|s^{-1}(v_1)|, |s^{-1}(v_2)|, \ldots , |s^{-1}(v_j)|, \ldots |s^{-1}(v_i)|, \ldots, |s^{-1}(v_n)|).\]
Let us denote the vertixes of $G$ by $\{w_1,\dots ,w_n\}$ and calculate:
\begin{displaymath}
\begin{split}
P_s^N(F)-P_s^N(G)&=
\left( \vert s^{-1}(v_i)\vert \vert s^{-1}(v_j) \vert \sum_{k=1}^{i} \prod_{\substack{m=k \\ m\not=i \\ m\not =j}}^{n-1} |s^{-1}(v_m)|  -\vert s^{-1}(w_i)\vert \vert s^{-1}(w_j) \vert \sum_{k=1}^{i} \prod_{\substack{m=k \\ m\not=i \\ m\not =j}}^{n-1} |s^{-1}(w_m)| \right) \\ 
&+\left(\vert s^{-1}(v_j)\vert \sum_{k=i+1}^{j} \prod_{\substack{m=k \\ m\not =j}}^{n-1} |s^{-1}(v_m)| -\vert s^{-1}(w_j)\vert \sum_{k=i+1}^{j} \prod_{\substack{m=k \\ m\not =j}}^{n-1} |s^{-1}(w_m)| \right) \\
&+\left( \sum_{k=j+1}^{n-1} \prod_{m=k}^{n-1} |s^{-1}(v_m)| -\sum_{k=j+1}^{n-1} \prod_{m=k}^{n-1} |s^{-1}(w_m)|  \right) \\
&=\vert s^{-1}(v_j)\vert \sum_{k=i+1}^{j} \prod_{\substack{m=k \\ m\not =j}}^{n-1} |s^{-1}(v_m)| -\vert s^{-1}(v_i)\vert \sum_{k=i+1}^{j} \prod_{\substack{m=k \\ m\not =j}}^{n-1} |s^{-1}(v_m)|\\
&=\left( \vert s^{-1}(v_j)\vert-\vert s^{-1}(v_i)\vert \right) \sum_{k=i+1}^{j} \prod_{\substack{m=k \\ m\not =j}}^{n-1} |s^{-1}(v_m)| \ge 0
\end{split}
\end{displaymath}

The last inequality follows from $F$ being optimal.
Since $\prod\limits_{\substack{m=k \\ m\not =j}}^{n-1} |s^{-1}(v_m)|$ is always positive the inequality $\vert s^{-1}(v_j)\vert-\vert s^{-1}(v_i)\vert \ge 0$ must hold.
\end{proof}

\begin{corollary}\label{edgeseqsinks}
	All lemmas that we have proven in this case where our graph has a unique sink yield a very particular form of the sequence of edges $l_s:=(|s_F(v_i)^{-1}|)_{i=1}^{n-1}$ in an optimal graph. Namely, we know that all entries are either $1$, $2$ or $3$ by Lemma~\ref{lem4.1}. sequence is non-decreasing as read from left to right, and mad up of numbers $1,2,$ or $3$ So, in fact, $l_s$ is a sequence of the form
	\begin{equation*}
		l_s = (1,\dots,1,2,\dots,2,3,\dots,3).
	\end{equation*}
	Say that there are $n_1$ copies of $1$'s at the beginning, followed by $n_2$ $2$'s, and $n_3$ $3$'s. We will use the following shorthand for the sequence $l_s$:
	\begin{equation*}
		l_s =: (n_1,n_2,n_3)
	\end{equation*}
\end{corollary}

\begin{lemma}
For $E$ whose associated sequence is $(n_1,n_2,n_3)$ the number of paths to the sink is given by formula
\begin{displaymath}
P_s^N\!(E)=\frac{1}{2}\left( (2^{n_2+2}+n_1 2^{n_2+1}-1)3^{n_3}-1 \right)=(n_1+2)2^{n_2}3^{n_3}-\frac{3^{n_3}+1}{2}
\end{displaymath}
\end{lemma}

\begin{proof}
\begin{displaymath}
\begin{split}
P_s^N\!(E)
&=\sum\limits_{k=1}^{n-1} \prod\limits_{i=k}^{n-1} \vert s^{-1}(v_i) \vert +1\\
&=1+\sum\limits_{k=1}^{n-1} \prod\limits_{i=1}^k \vert s^{-1}(v_{n-i}) \vert \\
&=1+\sum\limits_{k=1}^{n_3} \prod\limits_{i=1}^k \vert s^{-1}(v_{n-i}) \vert+\sum\limits_{k=n_3+1}^{n_3+n_2} \prod\limits_{i=1}^k \vert s^{-1}(v_{n-i})\vert+\sum\limits_{k=n_3+n_2+1}^{n_3+n_2+n_1}\prod\limits_{i=1}^k \vert s^{-1}(v_{n-i})\vert\\
&=1+\sum\limits_{k=1}^{n_3} 3^k +\sum\limits_{k=1}^{n_2} 3^{n_3}2^k+\sum\limits_{k=1}^{n_1} 3^{n_3}2^{n_2} \\
&=\sum\limits_{k=0}^{n_3} 3^k +2\cdot 3^{n_3}\sum\limits_{k=0}^{n_2-1} 3^{n_3}2^k+\sum\limits_{k=1}^{n_1} 3^{n_3}2^{n_2} \\
&=\frac{3^{n_3+1}-1}{3-1}+2\cdot 3^{n_3}\frac{2^{n_2}-1}{2-1}+n_1\cdot 2^{n_2}\cdot 3^{n_3}\\
&=\frac{1}{2}\left( 3^{n_3+1}+4\cdot 2^{n_2} \cdot 3^{n_3}+2n_1 2^{n_2} 3^{n_3}-4\cdot 3^{n_3}-1 \right) \\
&=\frac{1}{2}\left( \left( 2^{n_2+2}+n_1 2^{n_2+1}-1\right) 3^{n_3} -1\right)
\end{split}
\end{displaymath}
\end{proof}

\begin{lemma}\label{lem4.2}
Let us consider a trunk with associated sequence $(n_1,n_2,n_3)$.
\begin{enumerate}
\item If $n_1 \ge 2$ then transformation $(n_1,n_2,n_3)\mapsto (n_1-2,n_2+1,n_3)$ does not decrease the number of paths to the sink.
\item If $n_2\ge 4$ then transformation $(n_1,n_2,n_3)\mapsto (n_1,n_2-3,n_3+2)$ does not decrease the number of paths to the sink.
\item If $n_1=0$ and $n_2 \ge 2$ then transformation $(n_1,n_2,n_3)\mapsto (n_1+1,n_2-2,n_3+1)$ does not decrease the number of paths to the sink.
\item If $n_1=0$, $n_3 \ge 1$ then transformation $(n_1,n_2,n_3)\mapsto (n_1+1,n_2+1,n_3-1)$ does not decrease the number of paths to the sink.
\end{enumerate}
\end{lemma}

\begin{proof}
\begin{enumerate}
\item We calculate the difference of number of paths to the sink:
\begin{equation*}
\begin{split}
&P_s^N\!(n_1-2,n_2+1,n_3)-P_s^N\!(n_1,n_2,n_3) \\
&=n_12^{n_2+1}3^{n_3}-\frac{3^{n_3}+1}{2}-\left( (n_1+2)2^{n_2}3^{n_3}-\frac{3^{n_3}+1}{2} \right) \\
&=(n_1-2)2^{n_2}3^{n_3} \ge 0
\end{split}
\end{equation*}
The last inequality holds because $n_1\ge 2$
\item We calculate the difference of number of paths to the sink:
\begin{equation*}
\begin{split}
&P_s^N\!(n_1,n_2-3,n_3+2)-P_s^N\!(n_1,n_2,n_3) \\
&=(n_1+2)2^{n_2-3}3^{n_3+2}-\frac{3^{n_3+3}+1}{2}-\left( (n_1+2)2^{n_2}3^{n_3}-\frac{3^{n_3}+1}{2} \right) \\
&= \left( (n_1+2)2^{n_2-3}-4 \right) 3^{n_3} \ge 0
\end{split}
\end{equation*}
The last inequality holds because $n_2\ge 4$.
\item We calculate the difference of number of paths to the sink:
\begin{equation*}
\begin{split}
&P_s^N\!(n_1+1,n_2-2,n_3+1)-P_s^N\!(n_1,n_2,n_3) \\
&=(n_1+3)2^{n_2-2}3^{n_3+1}-\frac{3^{n_3+1}+1}{2}-\left( (n_1+2)2^{n_2}3^{n_3}-\frac{3^{n_3}+1}{2} \right) \\
&=\left( (1-n_1)2^{n_2-2}-1 \right) 3^{n_3} \ge 0.
\end{split}
\end{equation*}
The last inequality holds because $n_1=0$, $n_2\ge 2$
\item We calculate the difference of number of paths to the sink:
\begin{equation*}
\begin{split}
&P_s^N\!(n_1+1,n_2+1,n_3-1)-P_s^N\!(n_1,n_2,n_3) \\
&=(n_1+3)2^{n_2+1}3^{n_3-1}-\frac{3^{n_3-1}+1}{2}-\left( (n_1+2)2^{n_2}3^{n_3}-\frac{3^{n_3}+1}{2} \right) \\
&=\left( 1-n_12^{n_2} \right) 3^{n_3-1} \ge 0.
\end{split}
\end{equation*}
The last inequality holds because $n_1=0$, $n_3\ge 1$
\end{enumerate}
\end{proof}

\begin{theorem}\label{th4.0}
Let $N\ge 5$. The supremum $\sup_{E\in\mathcal{T}^N_3}P_s^N\!(E)$ is realised by a trunk whose associated sequence is either $\left( 1,1,\left\lfloor \frac{N}{3}\right\rfloor -1\right)$, $\left( 1,3,\left\lfloor \frac{N}{3}\right\rfloor -2\right)$, $\left( 1,2,\left\lfloor \frac{N}{3}\right\rfloor -1\right)$, depending on the divisibility of $N$ by 3.
\end{theorem}

\begin{proof}
Let us start with any trunk which has $N$ edges. 
Let $(n_1,n_2,n_3)$ be its associated sequence.
Let us apply transformation $1$ from lemma \ref{lem4.2} as many times as possible and receive graph whose sequence is $\left( n_1-2\left\lfloor \frac{N}{2}\right\rfloor,n_2+\left\lfloor \frac{n_1}{2}\right\rfloor, n_3 \right)$.
Now there are two cases:\\
If $n_1$ is even then either $n_2+\frac{n_1}{2} \ge 2$ or $n_3\ge 1$, because $2\left( n_2+\frac{n_1}{2} \right)+3n_3\ge 5$.
Depending on the case we apply transformation $3$ or $4$ from lemma \ref{lem4.2} to get a graph whose sequence is of type $(1,m_2,m_3)$.\\
If $n_1$ is odd then sequence $\left( n_1-2\left\lfloor \frac{N}{2}\right\rfloor,n_2+\left\lfloor \frac{n_1}{2}\right\rfloor, n_3 \right)$ is of the form $(1,m_2,m_3)$. \\
Finally to the sequence $(1,m_2,m_3)$ we apply transformation $2$ from lemma \ref{lem4.2} as many times as possible and get one of the a graphs whose sequences have been stated in the theorem.
\end{proof}

\subsection{Three Sisters}
\begin{theorem}
For $N \geq 5$ where $N$ is the number of edges, the following formulas compute $SP^N_*$, 
the maximal number of paths ending in the sink for a graph with $N=3k+i$, where $i \in \{0,1,2 \}$:
\phantom{.}\\
\begin{center}
\setlength\extrarowheight{16pt}
\boxed{
\begin{tabular}{lll}
$i=0$&& $SP^N_*=\displaystyle\frac{1}{2}\left( 11\cdot 3^{k-1}-1\right)$\\
$i=1$&& $SP^N_*=\displaystyle\frac{1}{2}\left( 47\cdot 3^{k-2}-1 \right)$\\
$i=2$ && $SP^N_*=\displaystyle\frac{1}{2}\left( 23\cdot 3^{k-1}-1\right)$\\
\end{tabular}}
\end{center}
\end{theorem}
\begin{proof}
We simply compute the formula from lemma \ref{lem4.05} to the graphs from theorem \ref{th4.0}.
\end{proof}
Asymptotically, $\displaystyle \frac{SP^\infty_*}{P^\infty_*} > \frac{1}{2}$.

The matrix algebra $M_n(k)$ is the Leavitt path algebra $L_k(E)$ of a one-sink acyclic graph~$E$. 
Three Sisters tell us the maximal size $n$ (dimension $n^2$) of the
matrix algebra that can be obtained as the Leavitt path algebra of $E$ given a fixed amount $N$ of edges in~$E$. 
However, one can also ask the question: 
What is the minimal amount of edges needed in a graph to obatin the matrix algebra of a fixed size as the Leavitt path algebra of this graph?
The following corollary gives us an estimate:
\begin{corollary} \label{necessary}
Let $L_k(E)\cong M_n(k)$. Then the  amount of edges in $E$ is at least
$$ 
N_n:=\max\{ N \in \mathbb{N} \, | \,  n \geq SP^N_*\}.  
$$ 
\end{corollary}

\begin{remark}\label{Remark}
In general, it is not true that  $N_n+1$ many edges suffice to construct a graph $E$ such that $L_k(E) \cong M_n(k)$. Indeed,  
$N_{12}=N_{13}=N_{14}=N_{15}=5$ and 6 edges suffice to construct graphs whose Leavitt path algebras are, respectively, the matrix algebras of size
12, 13 and 15,  but one needs at least 7 edges to construct a graph $E$ such that $L_k(E) \cong M_{14}(k)$. 
The latter can  be proven by a direct classification of  one-sink graphs of weight 6  and finding an appropriate graph of weight~7:
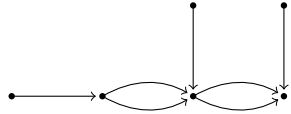
\begin{center}
\begin{tikzpicture}[scale=1.2]
\fill (0,0)  circle[radius=1pt];
\fill (1,0)  circle[radius=1pt];
\fill (2,0)  circle[radius=1pt];
\fill (3,0)  circle[radius=1pt];
\fill (2,1)  circle[radius=1pt];
\fill (3,1)  circle[radius=1pt];
\draw[->, shorten >=2.5pt]  (0,0) to (1,0);
\draw[->, shorten >=2.5pt] (1,0) to[in=150, out=30] (2,0);
\draw[->, shorten >=2.5pt] (1,0) to[in=210, out=-30] (2,0);
\draw[->, shorten >=2.5pt] (2,0) to[in=150, out=30] (3,0);
\draw[->, shorten >=2.5pt] (2,0) to[in=210, out=-30] (3,0);
\draw[->, shorten >=2.5pt] (2,1) to (2,0);
\draw[->, shorten >=2.5pt] (3,1) to (3,0);

\end{tikzpicture}
\end{center}
\end{remark}

Corollary~\ref{necessary} computes the minimal amount of edges \emph{necessary} to build a matrix of size $n$ as a Leavitt path algebra.
Let us now define the minimal amount of edges \emph{sufficient} to build a matrix of size $n$ as a Leavitt path algebra:
\[
N^n:=\min\{ N \in \mathbb{N} \, | \,  \exists\; \text{a graph}\; E_N\colon |E^1_N|=N\;\text{and}\;L_k(E_N)=M_n(k)\}.
\]
The above remark shows that $N^{14}=7$, so $\sup_{n\in\mathbb{N}\setminus\{0\}}\{N^n-N_n\}\geq 2$. This leads us to the following puzzling question:
\[
\sup_{n\in\mathbb{N}\setminus\{0\}}\{N^n-N_n\}<\infty\;?
\]

\section{Acknowledgements}

\bibliographystyle{acm}

\end{document}